\documentclass[a4paper,twoside]{article}
\usepackage{a4}
\usepackage{amssymb}
\usepackage{amsmath}
\usepackage{upref}
\usepackage{xypic}
\usepackage[active]{srcltx}
\usepackage[dvips,colorlinks,citecolor=blue,linkcolor=blue]{hyperref}
\usepackage[dvipsnames]{color}
\allowdisplaybreaks[2] 
\newcount\minutes \newcount\hours
\hours=\time
\divide\hours 60
\minutes=\hours
\multiply\minutes -60
\advance\minutes \time
\newcommand{\klockan}{\the\hours:{\ifnum\minutes<10 0\fi}\the\minutes}
\newcommand{\tid}{\today\ \klockan}
\newcommand{\prtid}{\smash{\raise 10mm \hbox{\LaTeX ed \tid}}}
\renewcommand{\prtid}{}
\makeatletter
\def\sectionmark#1{} 
\def\subsectionmark#1{}
\newcommand{\sectnr}{\ifnum \c@secnumdepth >\z@
                 \thesection.\hskip 1em\relax \fi}
\def\@evenhead{\footnotesize\rm\thepage\hfil\leftmark\hfil\llap{\prtid}}
\def\@oddhead{\footnotesize\rm\rlap{\prtid}\hfil\rightmark\hfil\thepage}
\def\tableofcontents{\section*{Contents} 
 \@starttoc{toc}}
\makeatother
\makeatletter
\def\@biblabel#1{#1.}
\makeatother
\makeatletter
\let\Thebibliography=\thebibliography
\renewcommand{\thebibliography}[1]{\def\@mkboth##1##2{}\Thebibliography{#1}
\addcontentsline{toc}{section}{References}
\frenchspacing 
\setlength{\@topsep}{0pt}
\setlength{\itemsep}{0pt}%
\setlength{\parskip}{0pt plus 2pt}%
}
\makeatother
\makeatletter
\def\mdots@{\mathinner.\nonscript\!.%
 \ifx\next,.\else\ifx\next;.\else\ifx\next..\else
 \nonscript\!\mathinner.\fi\fi\fi}
\let\ldots\mdots@
\let\cdots\mdots@
\let\dotso\mdots@
\let\dotsb\mdots@
\let\dotsm\mdots@
\let\dotsc\mdots@
\def\vdots{\vbox{\baselineskip2.8\p@ \lineskiplimit\z@
    \kern6\p@\hbox{.}\hbox{.}\hbox{.}\kern3\p@}}
\def\ddots{\mathinner{\mkern1mu\raise8.6\p@\vbox{\kern7\p@\hbox{.}}%
    \raise5.8\p@\hbox{.}\raise3\p@\hbox{.}\mkern1mu}}
\makeatother
\makeatletter
\let\Enumerate=\enumerate
\renewcommand{\enumerate}{\Enumerate%
\setlength{\@topsep}{0pt}
\setlength{\itemsep}{0pt}%
\setlength{\parskip}{0pt plus 1pt}%
\renewcommand{\theenumi}{\textup{(\alph{enumi})}}%
\renewcommand{\labelenumi}{\theenumi}%
}
\let\endEnumerate=\endenumerate
\renewcommand{\endenumerate}{\endEnumerate\unskip}
\makeatother
\makeatletter
\def\@seccntformat#1{\csname the#1\endcsname.\quad}
\makeatother
\newcommand{\authortitle}[2]{\author{#1}\title{#2}\markboth{#1}{#2}}
\newcommand{\art}[6]{{\sc #1, \rm #2, \it #3\/ \bf #4 \rm (#5), \mbox{#6}.}}
\newcommand{\auth}[2]{{#1, #2.}}
\newcommand{\artprep}[3]{{\sc #1, \rm #2, \it #3.}}
\newcommand{\artin}[3]{{\sc #1, \rm #2, in #3.}}
\newcommand{\book}[3]{{\sc #1, \it #2, \rm #3.}}
\newcommand{\AND}{{\rm and }}
\RequirePackage{amsthm}
\newtheoremstyle{descriptive}%
  {\topsep}   
  {\topsep}   
  {\rmfamily} 
  {}          
  {\bfseries} 
  {.}         
  { }         
  {}          
\newtheoremstyle{propositional}%
  {\topsep}   
  {\topsep}   
  {\itshape}  
  {}          
  {\bfseries} 
  {.}         
  { }         
  {}          
\theoremstyle{propositional}
\newtheorem{thm}{Theorem}[section]
\newtheorem{prop}[thm]{Proposition}
\newtheorem{lem}[thm]{Lemma}
\theoremstyle{descriptive}
\newtheorem{deff}[thm]{Definition}
\newtheorem{example}[thm]{Example}
\newtheorem{remark}[thm]{Remark}
\makeatletter
\renewenvironment{proof}[1][\proofname]{\par
  \pushQED{\qed}%
  \normalfont
  \trivlist
  \item[\hskip\labelsep
        \itshape
    #1\@addpunct{.}]\ignorespaces
}{%
  \popQED\endtrivlist\@endpefalse
}
\makeatother
\newcommand{\setm}{\setminus}
\renewcommand{\subsetneq}{\varsubsetneq}
\renewcommand{\emptyset}{\varnothing}
\def\vint{\mathop{\mathchoice%
          {\setbox0\hbox{$\displaystyle\intop$}\kern 0.22\wd0%
           \vcenter{\hrule width 0.6\wd0}\kern -0.82\wd0}%
          {\setbox0\hbox{$\textstyle\intop$}\kern 0.2\wd0%
           \vcenter{\hrule width 0.6\wd0}\kern -0.8\wd0}%
          {\setbox0\hbox{$\scriptstyle\intop$}\kern 0.2\wd0%
           \vcenter{\hrule width 0.6\wd0}\kern -0.8\wd0}%
          {\setbox0\hbox{$\scriptscriptstyle\intop$}\kern 0.2\wd0%
           \vcenter{\hrule width 0.6\wd0}\kern -0.8\wd0}}%
          \mathopen{}\int}
\newcommand{\grad}{{\nabla}}
\DeclareMathOperator{\dvg}{div}
\DeclareMathOperator{\Lip}{Lip}
\newcommand{\Lipc}{{\Lip_c}}
\DeclareMathOperator{\spt}{supp}
\newcommand{\supp}{\spt}
\DeclareMathOperator*{\essliminf}{ess\,lim\,inf}
\DeclareMathOperator*{\essinf}{ess\,inf}
\DeclareMathOperator*{\esssup}{ess\,sup}
\newcommand{\bdry}{\partial}
\newcommand{\bdy}{\bdry}
\newcommand{\loc}{_{\rm loc}}
{\catcode`p =12 \catcode`t =12 \gdef\eeaa#1pt{#1}}      
\def\accentadjtext#1{\setbox0\hbox{$#1$}\kern   
                \expandafter\eeaa\the\fontdimen1\textfont1 \ht0 }
\def\accentadjscript#1{\setbox0\hbox{$#1$}\kern 
                \expandafter\eeaa\the\fontdimen1\scriptfont1 \ht0 }
\def\accentadjscriptscript#1{\setbox0\hbox{$#1$}\kern   
                \expandafter\eeaa\the\fontdimen1\scriptscriptfont1 \ht0 }
\def\accentadjtextback#1{\setbox0\hbox{$#1$}\kern       
                -\expandafter\eeaa\the\fontdimen1\textfont1 \ht0 }
\def\accentadjscriptback#1{\setbox0\hbox{$#1$}\kern     
                -\expandafter\eeaa\the\fontdimen1\scriptfont1 \ht0 }
\def\accentadjscriptscriptback#1{\setbox0\hbox{$#1$}\kern 
                -\expandafter\eeaa\the\fontdimen1\scriptscriptfont1 \ht0 }
\def\itoverline#1{{\mathsurround0pt\mathchoice
        {\rlap{$\accentadjtext{\displaystyle #1}
                \accentadjtext{\vrule height1.593pt}
                \overline{\phantom{\displaystyle #1}
                \accentadjtextback{\displaystyle #1}}$}{#1}}
        {\rlap{$\accentadjtext{\textstyle #1}
                \accentadjtext{\vrule height1.593pt}
                \overline{\phantom{\textstyle #1}
                \accentadjtextback{\textstyle #1}}$}{#1}}
        {\rlap{$\accentadjscript{\scriptstyle #1}
                \accentadjscript{\vrule height1.593pt}
                \overline{\phantom{\scriptstyle #1}
                \accentadjscriptback{\scriptstyle #1}}$}{#1}}
        {\rlap{$\accentadjscriptscript{\scriptscriptstyle #1}
                \accentadjscriptscript{\vrule height1.593pt}
                \overline{\phantom{\scriptscriptstyle #1}
                \accentadjscriptscriptback{\scriptscriptstyle #1}}$}{#1}}}}
\newcommand{\de}{\delta}
\newcommand{\eps}{\varepsilon}
\newcommand{\Om}{\Omega}
\newcommand{\clOmprime}{{\overline{\Om}\mspace{1mu}}'}
\renewcommand{\phi}{\varphi}
\newcommand{\R}{\mathbf{R}}
\newcommand{\Rn}{{\R^n}}
\newcommand{\Q}{\mathbf{Q}}
\newcommand{\limminus}{{\mathchoice{\raise.17ex\hbox{$\scriptstyle -$}}
                {\raise.17ex\hbox{$\scriptstyle -$}}
                {\raise.1ex\hbox{$\scriptscriptstyle -$}}
                {\scriptscriptstyle -}}}
\newcommand{\limplus}{{\mathchoice{\raise.17ex\hbox{$\scriptstyle +$}}
                {\raise.17ex\hbox{$\scriptstyle +$}}
                {\raise.1ex\hbox{$\scriptscriptstyle +$}}
                {\scriptscriptstyle +}}}
\newcommand{\oHp}{H}                
\newcommand{\K}{\mathcal{K}}%
\newcommand{\ut}{\tilde{u}}
\newcommand{\ft}{\tilde{f}}
\numberwithin{equation}{section}
\newcommand{\eqv}{\mathchoice{\quad \Longleftrightarrow \quad}{\Leftrightarrow}
                {\Leftrightarrow}{\Leftrightarrow}}
\newcommand{\imp}{\mathchoice{\quad \Longrightarrow \quad}{\Rightarrow}
                {\Rightarrow}{\Rightarrow}}
\newenvironment{ack}{\medskip{\it Acknowledgement.}}{}
\newcommand{\px}{{\ensuremath{p(\cdot)}}}
\newcommand{\p}{\ensuremath{p(x)}} 
\newcommand{\pp}{{$p\mspace{1mu}$}}   
\newcommand{\pplus}{{p^\limplus}}
\newcommand{\pminus}{{p^\limminus}}
\newcommand{\Lpx}{{L^{p(\cdot)}}}
\newcommand{\Wpx}{W^{1,p(\cdot)}}
\newcommand{\Wpxloc}{W^{1,p(\cdot)}\loc}
\newcommand{\Cpx}{{C_{p(\cdot)}}}

\begin{document}

%
%
\authortitle{Tomasz Adamowicz, Anders Bj\"orn and Jana Bj\"orn}
{\px-superharmonic functions, the Kellogg property and semiregular points}
\title{Regularity of \px-superharmonic functions, the Kellogg property and semiregular boundary points}

\author{
Tomasz Adamowicz
\\
\it\small Department of Mathematics, Link\"opings universitet, \\
\it\small SE-581 83 Link\"oping, Sweden\/{\rm ;}
\it \small tomasz.adamowicz@liu.se
\\
\\
Anders Bj\"orn \\
\it\small Department of Mathematics, Link\"opings universitet, \\
\it\small SE-581 83 Link\"oping, Sweden\/{\rm ;}
\it \small anders.bjorn@liu.se
\\
\\
Jana Bj\"orn \\
\it\small Department of Mathematics, Link\"opings universitet, \\
\it\small SE-581 83 Link\"oping, Sweden\/{\rm ;}
\it \small jana.bjorn@liu.se
}

\date{} 
\maketitle

\noindent{\small
{\bf Abstract}.
We study various boundary and inner
regularity questions for $\px$-(super)har\-monic functions in Euclidean domains.
In particular, we prove the Kellogg property and introduce 
a classification of 
boundary points for $\px$-harmonic functions
into three disjoint classes: regular, semiregular and strongly irregular points. 
Regular and especially semiregular points are characterized in many ways.
The discussion is illustrated by examples.

Along the way, we present a removability result for bounded $\px$-harmonic 
functions and give some new characterizations of $W^{1, \px}_0$ spaces.
We also show that $\px$-superharmonic functions are lower semicontinuously regularized, and characterize them in terms of lower semicontinuously regularized
supersolutions.
}

\bigskip
\noindent
{\small \emph{Key words and phrases}: 
comparison principle,
Kellogg property, 
lsc-regularized,
nonlinear potential theory,
nonstandard growth equation,
obstacle problem,
\px-harmonic, 
quasicontinuous,
regular boundary point, 
removable singularity, 
semiregular point, 
Sobolev space,
strongly irregular point,
\px-superharmonic, 
\px-supersolution,
trichotomy, 
variable exponent.
} 



\section{Introduction}

The theory of partial differential equations with nonstandard growth has been 
a subject of increasing interest in the last decade. 
Several results known for the model elliptic differential operator of nonlinear 
analysis, the \pp-Laplacian $\Delta_p:=\dvg(|\grad u|^{p-2}\grad u)$, 
have been established in the variable exponent setting for 
the so-called $\px$-Laplace equation and some of its modifications. 
The $\px$-Laplace equation 
\[
\dvg(p(x)|\grad u|^{p(x)-2}\grad u) = 0
\] 
is the Euler--Lagrange equation for the minimization of the 
$\px$-Dirichlet integral 
\[
\int_\Om |\grad u|^{p(x)}\,dx
\]
among functions with given boundary data.
Such minimization problems and equations arise for instance from applications 
in image processing, see Chen--Levine--Rao~\cite{CLR}, and  in the description 
of electrorheological 
fluids, see Acerbi--Mingione~\cite{AM02} and
 R\r u\v zi\v cka~\cite{Ru}. 

Variable exponent equations have been studied, among others, in the context of 
interior regularity  of solutions, see e.g.\ Acerbi--Mingione~\cite{AM4}, 
Fan~\cite{Fan07} and Henriques~\cite{Hen08}, 
and from the point of view of geometric properties of the solutions, 
see e.g.\ Adamowicz--H\"ast\"o~\cite{AdH09}, \cite{AdH10}. 
Also, the nonlinear potential theory 
associated with variable exponent elliptic equations 
has recently attracted attention, see e.g.\  
Harjulehto--Kinnunen--Lukkari~\cite{HaKiLu07},  
Harjulehto--H\"ast\"o--Koskenoja--Lukkari--Marola~\cite{HHKLM07}, 
Latvala--Lukkari--Toivanen~\cite{LLT} and Lukkari~\cite{lukkari09}.
For a survey of recent results in the field we refer to 
Harjulehto--H\"ast\"o--L\^e--Nuortio~\cite{HHLN}. 

Despite the symbolic similarity to the \pp-Laplacian, various unexpected 
phenomena can occur when the exponent is a function, for instance the minimum 
of the $\px$-Dirichlet energy may not exist even in the one-dimensional case 
for smooth $p$, see~\cite[Section~3]{HHLN}, and smooth functions 
need not be dense in the
corresponding variable exponent Sobolev spaces, see the monograph 
by Diening--Harjulehto--H\"ast\"o--R\r{u}\v{z}i\v{c}ka~\cite[Chapter~9.2]{DHHR}.

 In this paper we address several questions regarding boundary regularity 
of $\px$-harmonic functions, i.e.\ the solutions of the $\px$-Laplace equation.
Our focus is on discussing various types of boundary points 
and on analyzing the structure of sets of such points.

A boundary point $x_0 \in \bdy \Om$ 
is \emph{regular} if
\[ 
         \lim_{\Om \ni y \to x_0} \oHp f(y)=f(x_0)
         \quad \text{for all } f\in C(\bdry\Om), 
\] 
where $\Om$ is a nonempty bounded open subset of $\R^n$ and
$\oHp f$ is the solution of the $\px$-Dirichlet problem 
with boundary values $f$.
(See later sections for notation and precise definitions.)

\begin{thm}\label{thm-kellogg}
\textup{(The Kellogg property)}
The set of all irregular boundary points 
has zero \px-capacity.
\end{thm}

The Kellogg property for variable exponents was recently obtained 
by Latvala--Lukkari--Toivanen~\cite{LLT}
using balayage 
and the Wiener criterion 
(the latter being due to Alkhutov--Krasheninnikova~\cite[Theorem~1.1]{Alkhutov-K}.)
Here we provide a shorter and more elementary proof, 
which in particular does not depend on the Wiener criterion.
It is based on the ideas introduced by 
Bj\"orn--Bj\"orn--Shanmugalingam~\cite{BBS}
for their proof of the Kellogg property in metric spaces 
(with constant $p$).
The proof in \cite{BBS} is based on Newtonian-type Sobolev spaces,
but
here we have refrained from the Newtonian approach and only use the usual 
variable exponent Sobolev spaces.
Our proof may therefore be of interest also in the constant $p$ case,
for readers who prefer to avoid Newtonian spaces.

That a boundary point is regular  can be rephrased 
in the following way.
A point $x_0 \in \bdy \Om$ is regular if the following two conditions hold:
\begin{enumerate}
\item \label{semi-intro}
for all $f \in C(\bdy \Om)$ the limit
\begin{equation} \label{eq-semi}
     \lim_{\Om \ni y \to x_0} \oHp f(y)
        \quad \text{exists};
\end{equation}
\item \label{strong-intro}
for all $f \in C(\bdy \Om)$ there is
a sequence $\{y_j\}_{j=1}^\infty$ such that
\begin{equation} \label{eq-strong}
\Om \ni y_j \to x_0 \text{ and } \oHp f(y_j) \to f(x_0),
\quad \text{as } j \to \infty.
\end{equation}
\end{enumerate}
\bigskip
%
It turns out that for irregular boundary points
\emph{exactly one} of these two properties holds,
i.e.\ it can never happen that both fail.
This is the content of the following theorem.
We say that  $x_0 \in \bdy \Om$ is 
\emph{semiregular} if \ref{semi-intro}
holds but not \ref{strong-intro},
and \emph{strongly irregular}
if \ref{strong-intro}
holds but not \ref{semi-intro}.

\begin{thm} \label{thm-trichotomy}
\textup{(Trichotomy)}
A boundary point $x_0 \in \bdy \Om$
is either regular, semiregular or 
strongly irregular.
\end{thm}

The first example (for $p=2$) of an irregular boundary point was given by 
Zaremba~\cite{zaremba} in 1911, in which he showed that the centre of
a punctured disk is irregular. 
This is an example of a semiregular point.
Shortly afterwards, Lebesgue~\cite{lebesgue1912} presented his famous
Lebesgue spine, whose tip is a
strongly irregular point
(see e.g.\ Remark~6.6.17 in 
Armitage--Gardiner~\cite{AG}).

In the linear case  the trichotomy was developed
in detail in Luke\v{s}--Mal\'y~\cite{lukesmaly} (in an axiomatic setting),
whereas  in the nonlinear constant-$p$ case
it was first stated by A.~Bj\"orn~\cite{ABclass} who obtained 
it in metric spaces and also for quasiminimizers.
As in \cite{ABclass}, there are two main ingredients needed to
obtain the trichotomy in the variable exponent case: 
the Kellogg property above and the following 
new removability result.

\begin{thm} \label{thm-removability-qharm-intro}
Let $F \subset \Om$ be relatively closed and such that $\Cpx(F)=0$.
If $u$ is a bounded \px-harmonic function in\/ $\Om \setm F$,
then it has a unique \px-harmonic extension to\/ $\Om$.
\end{thm}

Here and in Theorem~\ref{thm-superh-char-bdd},
$\Om$ is allowed to be unbounded.

The paper is organized as follows. 
In Section~\ref{sect-prelim} 
we recall some of the basic definitions and theorems from the 
theory of variable exponent Sobolev spaces, as well as potential theory. 
We also observe that some of the characterizations of the $\px$-Sobolev spaces 
with zero boundary data discussed in~\cite{DHHR} can be improved, and these 
improvements turn out useful for our later results. 
We also discuss the ``squeezing'' Lemma~\ref{lem-police} for variable exponent 
Sobolev spaces with zero boundary values, which to our best knowledge was not 
known or formulated in the literature so far. 

In Section~\ref{sect-supersoln} we discuss \px-supersolutions and the obstacle problem in the variable 
exponent setting. 
We discuss the existence and uniqueness  of solutions to the obstacle problem 
and their regularized representatives.
In addition, we obtain comparison principles for Dirichlet and 
obstacle problems, see Lemmas~\ref{lem-comp-principle} and~\ref{lem-obst-le}.

Section~\ref{sect-superh} is devoted to studying \px-superharmonic functions. 
Although this notion is well known, also in the $\px$-setting, we establish 
the following
new characterization of bounded $\px$-superharmonic functions.

\begin{thm} \label{thm-superh-char-bdd}
Assume that $u:\Om \to \R$ is locally bounded from above in\/ $\Om$.
Then $u$ is \px-superharmonic if and only if it is an
lsc-regularized \px-supersolution.
\end{thm}

For unbounded \px-superharmonic functions
we obtain a similar (but necessarily more involved)
characterization in Theorem~\ref{thm-superh-supermin-char}.

Section~\ref{sect-cont-bdy-values} is devoted to the Kellogg property,
whereas the 
removability result (Theorem~\ref{thm-removability-qharm-intro} above)
is obtained in Section~\ref{sect-remove}.
In the latter section  we also obtain a similar removability
result for bounded \px-superharmonic functions.
Lukkari~\cite{lukkari09} studied removability for 
unbounded \px-harmonic functions, our results are however not included in
his treatment.

In Section~\ref{sect-bdry-reg} we obtain the
trichotomy (Theorem~\ref{thm-trichotomy}) and 
also provide a number of characterizations of regular points.
In the last section we 
focus on semiregularity and give several characterizations
both of semiregular points themselves  and of sets of semiregular points,
involving capacity and  \px-harmonic and \px-superharmonic extensions.
In particular, we show that semiregularity is a local property. 
A similar result for regular 
points is a direct
consequence of the Wiener
criterion.
It would be interesting to obtain the locality for regular 
(and thus also for strongly irregular)  points more 
directly, without appealing to the Wiener criterion.
Let us again stress the fact that 
we do not use the Wiener criterion in this paper, except
for constructing a few examples in Example~\ref{ex-punctured-ball} 
and Propositions~\ref{prop-ball-minus-K}  and~\ref{prop-K1-K2}.

\begin{ack}
A.~B. and J.~B. were supported by the Swedish Research Council.
\end{ack}

\section{Preliminaries}
\label{sect-prelim}

A \emph{variable exponent} is 
a measurable function $p\colon \R^n\to [1,\infty]$.
In this paper we assume that
\[
     1 < \pminus \le \pplus < \infty,
     \quad \text{where }\quad \pminus=\essinf_{\Rn} p \quad
     \text{ and }\quad \pplus=\esssup_{\Rn} p,
\]
and that $p$ is log-H\"older continuous, i.e.\ 
there is a constant
$L>0$ such that
\begin{equation}  \label{eq-def-log-Holder}
|p(x)-p(y)|\leq \frac{L}{\log(e+1/|x-y|)}
 \quad \text{for } x,y\in \Rn.
\end{equation}
In addition, one usually assumes that $p$ satisfies the log-H\"older 
decay condition 
(see Definition~4.1.1 and the discussion in Chapter~4.1 in 
Diening--Harjulehto--H\"ast\"o--R\r u\v zi\v cka~\cite{DHHR}). 
However, 
for the results in this paper no decay condition is required.
We also assume throughout the paper that
$\Omega\subset\R^n$ is a nonempty open set. 
(In Sections~\ref{sect-cont-bdy-values}--\ref{sect-semi} as well as
in the second half of Section~\ref{sect-supersoln}
we will further assume that $\Om$ is bounded.)
For background on variable exponent function spaces we refer to 
\cite{DHHR}.

The \emph{variable
exponent Lebesgue space $L^{p(\cdot)}(\Omega)$} consists of all
measurable functions $u\colon \Omega\to \R$ for which the so-called Luxemburg norm
\[
\|u\|_{L^{p(\cdot)}(\Omega)}:= \inf\biggl\{\lambda > 0\,\colon\,
\int_{\Omega} \left|\frac{u(x)}{\lambda}\right|^{p(x)}dx\leq 1\biggr\}
\]
is finite.
Equipped with this norm, $L^{p(\cdot)}(\Omega)$ becomes a Banach space.
The variable exponent Lebesgue space is a special case of a
Musielak--Orlicz space. For a constant function $p$, it coincides
with the standard Lebesgue space. 

One of the difficulties when extending results from the constant to 
variable exponent setting is the lack of functional 
relationship between the norm and the integral. 
Nevertheless,  we 
do have the following useful estimates 
\begin{equation}\label{norm-mod-le-1}
\left(\int_{\Omega} |u(x)|^{p(x)}\,dx\right)^{1/\pminus} \leq \| u\|_{L^{p(\cdot)}(\Omega)} 
\leq \left(\int_{\Omega} |u(x)|^{p(x)}\,dx\right)^{1/\pplus}
\end{equation}
whenever $\int_{\Omega} |u(x)|^{p(x)}\,dx \le 1$.
For a proof and further discussion we refer to Lemmas~3.2.4
and~3.2.5 in  Diening--Harjulehto--H\"ast\"o--R\r u\v zi\v cka~\cite{DHHR}.
Note also, that if $\{u_i\}_{i=1}^{\infty}$ is a sequence of 
$L^{p(\cdot)}(\Omega)$-integrable functions, 
then from \eqref{norm-mod-le-1} we infer that
\begin{equation}\label{mod-norm-zero-conv}
\lim_{i\to \infty}\int_{\Omega} |u_i(x)|^{p(x)}\,dx=0 \eqv
 \lim_{i\to \infty} \| u_i\|_{L^{p(\cdot)}(\Omega)}=0.
\end{equation}

If $p\ge q$
are variable exponents, then $L\loc^{p(\cdot)}(\Om)$
embeds into $L\loc^{q(\cdot)}(\Om)$. 
In particular, 
every function in $L\loc^{p(\cdot)}(\Omega)$ also belongs to
$L\loc^{\pminus}(\Omega)$ (see Theorem~3.3.1 and the discussion in 
Section~3.3 in \cite{DHHR}). 
The H\"older inequality takes the form
\begin{equation}\label{Holder-ineq}
\int_\Omega u v \,dx 
   \le 2 \, \|u\|_{L^{p(\cdot)}(\Omega)} \|v\|_{L^{p'(\cdot)}(\Omega)},
\end{equation}
where $p'(\cdot)$ is the pointwise \textit{conjugate exponent}, i.e.\ $1/p(x)
+1/p'(x)\equiv 1$, (see Lemma~3.2.20 in \cite{DHHR}).

The \emph{variable exponent Sobolev space $W^{1,p(\cdot)}(\Omega)$}
consists of all $u\in L^{p(\cdot)}(\Omega)$ whose
distributional gradient $\nabla u$ also belongs to
$L^{p(\cdot)}(\Omega)$. The space
$W^{1,p(\cdot)}(\Omega)$ is a Banach space with the norm
\begin{displaymath}
\|u\|_{W^{1, p(\cdot)}(\Omega)}
     =\|u\|_{L^{p(\cdot)}(\Omega)}+\|\nabla u\|_{L^{p(\cdot)}(\Omega)}.
\end{displaymath}
In general, smooth functions are not dense in 
$\Wpx(\R^n)$ but the log-H\"older condition~\eqref{eq-def-log-Holder} 
guarantees that they are,
see Theorem~9.2.2 in \cite{DHHR}  and the discussion following it. 
We refer to Chapter~9 in~\cite{DHHR} 
for a detailed discussion of this topic.

\begin{deff}
 The \emph{\textup{(}Sobolev\/\textup{)} $\px$-capacity} of a set
$E\subset\R^n$ is defined as
\[
 C_{\px}(E):=\inf_{u}\int_{\R^n} (|u|^{\p}+|\nabla u|^{\p})\, dx,
\]
where the infimum is taken over all 
$u\in W^{1,\px}(\R^n)$ such that $u\geq 1$  in a neighbourhood of  $E$.
\end{deff}

The $\px$-capacity enjoys similar properties as 
in the constant case, see Theorem~10.1.2 in~\cite{DHHR}.
We say that a claim holds \emph{quasieverywhere} (\emph{q.e.} for short) 
if it holds everywhere except for a set with $\px$-capacity zero.

\begin{deff}
 A function $u:\Om\rightarrow [-\infty,\infty]$ is 
\emph{quasicontinuous} if for every $\varepsilon>0$ there exists 
an open set $U\subset \R^n$ with $C_{\px}(U)<\varepsilon$ 
such that $u|_{\Om\setminus U}$ is real-valued and continuous.
\end{deff}

Since $\Cpx$ is an outer capacity (which
follows directly from the definition) it is easy to show that
if $u$ is quasicontinuous and $v=u$ q.e., 
then $v$ is also quasicontinuous. 

 The following lemma sheds more light on quasicontinuous functions. It was 
obtained by Kilpel\"ainen for general capacities satisfying
two axioms, both of which are easily verified for the $\px$-capacity.

\begin{lem}\textup{(Kilpel\"ainen~\cite{Kil})} \label{lem-quasicont-ae-qe} 
If $u$ and $v$ are quasicontinuous in\/ $\Om$ and $u=v$ a.e.\ in\/ $\Om$,
then $u=v$ q.e.\ in\/ $\Om$.
\end{lem}

Following Definition~8.1.10 in 
Diening--Harjulehto--H\"ast\"o--R\r{u}\v{z}i\v{c}ka~\cite{DHHR}
we define \emph{the Sobolev space $W_0^{1,\px}(\Omega)$ with zero
boundary values} as the closure 
in $W^{1,\px}(\Omega)$ of $\Wpx(\R^n)$-functions with compact support in $\Om$.
By Proposition~11.2.3 in~\cite{DHHR}, this is equivalent to taking the
closure of $C_0^\infty(\Om)$ in $\Wpx(\Om)$.

In the rest of this section we give several useful characterizations
of $\Wpx_0(\Om)$ which will be needed later and do not seem to be
anywhere else in the literature.
The following result improves upon Theorem~11.2.6 in~\cite{DHHR},
where the same conclusion is obtained if $u\in\Wpx(\R^n)$
and $u=0$ q.e.\ in $\R^n\setm\Om$.

\begin{lem}   \label{lem-Wpx0-qe-on-bdry}
If $u\in\Wpx(\Om)$ is quasicontinuous in\/ $\R^n$  and $u=0$ q.e.\ on 
$\bdry\Om$, then $u\in\Wpx_0(\Om)$.
\end{lem}

\begin{proof}
By definition, we need to show that $u$ can be approximated in $\Wpx(\Om)$ by
functions from $\Wpx(\Om)$ with compact support in $\Om$.
This can be done in a similar way as the proof of Theorem~11.2.6 
in~\cite{DHHR}. 
Let us recall the main points of the argument.
By Lemma~9.1.1 in~\cite{DHHR}, $u$ can without loss of generality  be
assumed to be bounded and nonnegative.
Multiplying $u$ by the Lipschitz functions 
$\eta_j(x):=\min\{1,(j-|x|)_\limplus\}$ for $j=0,1,\ldots$ and noting that 
$\|u-u\eta_j\|_{\Wpx(\Om)}\to0$ as $j\to\infty$, we can also assume that
$u$ has bounded support.

Let $\eps>0$. By quasicontinuity and the fact that $u=0$ q.e.\ 
on $\bdry\Om$, there exists an open set $G\subset\R^n$ such that
$\Cpx(G)<\eps$, the restriction of $u$ to $\R^n\setm G$ is continuous and
$u=0$ on $\bdry\Om\setm G$.
In particular, this implies that the set 
\[
V:=\{x\in\R^n\setm G:u(x)<\eps\}
\]
is relatively open in $\R^n\setm G$ and $\bdry\Om\subset G\cup V$.
We can also find $w_\eps\in\Wpx(\R^n)$ such that $w_\eps=1$ on $G$,
$0\le w_\eps\le1$ in $\R^n$ and
\[
\int_{\R^n} ( |w_\eps|^{p(x)} + |\grad w_\eps|^{p(x)} )\,dx<\eps.
\]
As $G \cup V$ is open and contains $\bdy \Om$, 
it follows that the function $u_\eps:=(1-w_\eps)(u-\eps)_\limplus\chi_\Om$ has
compact support in $\Om$ and it is shown as in the proof of
Theorem~11.2.6 in~\cite{DHHR} that
$\|u-u_\eps\|_{\Wpx(\Om)}\to0$ as $\eps\to0$, i.e.\ $u\in\Wpx_0(\Om)$. 
\end{proof}

\begin{prop} \label{prop-Wpx0}
Assume that $u$ is quasicontinuous in\/ $\Om$.
Then $u \in \Wpx_0(\Om)$ if and only if
\[
      \ut:=\begin{cases}
         u & \text{in\/ } \Om, \\
         0 & \text{otherwise}, 
        \end{cases}
\]
is quasicontinuous and belongs to $\Wpx(\R^n)$.
\end{prop}

\begin{proof}
Assume first that $u \in \Wpx_0(\Om)$.
By Corollary~11.2.5 in~\cite{DHHR},
there is a quasicontinuous function $v \in \Wpx(\R^n)$
such that $v=u$ a.e.\ in $\Om$ and $v=0$ q.e.\ outside $\Om$.
By Lemma~\ref{lem-quasicont-ae-qe}, $v=u$ q.e.\ in $\Om$,
and thus $\ut=v$ q.e.\ in $\R^n$.
Hence $\ut \in \Wpx(\R^n)$ and $\ut$ is quasicontinuous.
The converse follows directly from Lemma~\ref{lem-Wpx0-qe-on-bdry}.
\end{proof}

The following ``squeezing lemma"
is useful when proving that certain functions 
belong to $\Wpx_0(\Om)$.

\begin{lem}  \label{lem-police}
Let $u\in\Wpx(\Om)$ and $u_1, u_2\in\Wpx_0(\Om)$ be such that
$u_1\le u\le u_2$ a.e.\ in\/ $\Om$.
Then $u\in\Wpx_0(\Om)$.
\end{lem}

We let $B(x,r)$ be the open ball with centre $x$ and radius $r$.

\begin{proof}
Replacing each function $v=u_1,u_2,u$ by its quasicontinuous representative 
\[
\limsup_{r\to0} \vint_{B(x,r)} v\,dx
\]
(provided by Theorem~11.4.4 in 
Diening--Harjulehto--H\"ast\"o--R\r{u}\v{z}i\v{c}ka~\cite{DHHR}),
we can assume that $u_1$, $u_2$ and $u$ are quasicontinuous in $\Om$
and $u_1\le u\le u_2$ everywhere in~$\Om$.

To be able to apply Lemma~\ref{lem-Wpx0-qe-on-bdry}, we need to show
that the zero extension of $u$ to $\R^n\setm\Om$ is quasicontinuous
in $\R^n$.
To this end,
Proposition~\ref{prop-Wpx0} implies that both $u_1$ and $u_2$ 
can be extended by zero outside $\Om$ to obtain quasicontinuous functions
on $\R^n$.
In other words, given $\eps>0$, there exists an open set $G$ with 
$\Cpx(G)<\eps$ such that the restrictions $u_1|_{\R^n\setm G}$ and
$u_2|_{\R^n\setm G}$ are continuous.  
Since $u|_{\R^n\setm G}$ lies between $u_1|_{\R^n\setm G}$ and $u_2|_{\R^n\setm G}$, and
$u=0$ on $\bdry\Om$, we conclude that $u|_{\R^n\setm G}$
is continuous at all $x\in\bdry\Om\setm G$.
It is clearly continuous in $\R^n\setm\overline{\Om}$ and quasicontinuous
in $\Om$.
Thus, $u$ is quasicontinuous in $\R^n$.
Lemma~\ref{lem-Wpx0-qe-on-bdry} then shows that $u\in\Wpx_0(\Om)$.
\end{proof}

\section{Supersolutions and obstacle problems}
\label{sect-supersoln}

In this section we include several auxiliary results about supersolutions 
and obstacle problems.
In particular, we discuss relations between these two notions, existence and 
uniqueness of the solutions, their interior regularity and a comparison
principle.
We shall consider the following type of obstacle problem.

\begin{deff} \label{deff-obst-K}
Let $f \in \Wpx(\Om)$ and $\psi \colon \Om \to [-\infty,\infty]$. 
Then we define
\[
    \K_{\psi,f}=\{v \in \Wpx(\Om) : v-f \in \Wpx_0(\Om) 
            \text{ and } v \ge \psi \ \text{a.e. in } \Om\}.
\]
A function $u \in \K_{\psi,f}$
is a \emph{solution of the $\K_{\psi,f}$-obstacle problem}
if 
\[
       \int_\Om |\nabla u|^{\p} \, dx 
       \le \int_\Om |\nabla v|^{\p} \, dx 
       \quad \text{for all } v \in \K_{\psi,f}.
\]
\end{deff}

The following equivalent definition of obstacle problems is given 
in~Harjulehto--H\"ast\"o--Koskenoja--Lukkari--Marola~\cite{HHKLM07},
p.\ 3427. The result in~\cite{HHKLM07} is obtained for a bounded $\Om$, but the proof is valid also 
for unbounded sets. In this paper, however, we will need it only for bounded sets.

\begin{prop}   \label{prop-sol-obst-iff-int}
The function  $u$ is a solution of the $\K_{\psi,f}$-obstacle problem
if and only if 
\[
       \int_\Om p(x)|\nabla u|^{\p-2} 
        \nabla u \cdot  \nabla (v-u)\, dx \ge 0
       \quad \text{for all } v \in \K_{\psi,f}.
\]
\end{prop}

\begin{deff} \label{deff-soln}
A function $u \in \Wpxloc(\Om)$ is a
\emph{\textup{(}super\/\textup{)}solution} of the $\px$-Laplace equation if
\[
      \int_{\phi \ne 0} |\nabla u|^{\p}\, dx 
      \le \int_{\phi \ne 0} |\nabla (u+\phi)|^{\p}\, dx 
\]
for all (nonnegative) $\phi \in C_0^\infty(\Om)$. A \emph{\px-harmonic function} is a continuous solution.
\end{deff}

Clearly, $u$ is a solution if and only if it 
is both a supersolution and a subsolution (i.e.\ $-u$ is a supersolution).
It is also immediate that a solution of an obstacle problem
is a supersolution.
Conversely, if $u$ is a supersolution in $\Om$ and $\Om'\Subset\Om$ is open
then by the density of $C^\infty_0(\Om')$ in $\Wpx_0(\Om')$ we see
that $u$ is a solution of the obstacle problem in $\Om'$ with $u$
as the obstacle and the boundary values.
(Recall that $A \Subset \Om$ if the closure of $A$ is a compact subset of $\Om$.)
The following characterization of (super)solutions then follows from 
Proposition~\ref{prop-sol-obst-iff-int}, cf.\ 
Harjulehto--H\"ast\"o--Koskenoja--Lukkari--Marola~\cite{HHKLM07},
p.\ 3427.
By the density of $C^\infty_0(\Om)$ again, it
is equivalent to require that \eqref{eq-prop-super-sol}
holds for all (nonnegative) $\phi \in \Wpx_0(\Om)$.

\begin{prop}\label{prop-super-sol}
A function $u \in \Wpxloc(\Om)$ is a\/
\textup{(}super\/\textup{)}solution if and only if
\begin{equation} \label{eq-prop-super-sol}
       \int_\Om p(x) |\nabla u|^{\p-2} \nabla u \cdot \nabla \phi\, dx \ge 0
\end{equation} 
for all\/ \textup{(}nonnegative\/\textup{)} $\phi \in C_0^\infty(\Om)$.
\end{prop}

For a function $u:\Om \to \R$, let
\[
    u^*(x)=\essliminf_{y \to x} u(y),
    \quad x \in \Om.
\]
It is easy to see that $u^*$ is always lower semicontinuous,
see the proof of Theorem~8.22 in 
Bj\"orn--Bj\"orn~\cite{BBbook}.
We call $u^*$ the \emph{lsc-regularization} of $u$,
and also say that $u$ is \emph{lsc-regularized} if $u=u^*$.

\begin{thm} \label{thm-u*}
Assume that $u$ is a supersolution in\/ $\Om$.
Then $u^*$ is a quasicontinuous supersolution in\/ $\Om$
and $u^*=u$ a.e.\ in $\Om$. 
Moreover, if $u$ is quasicontinuous, then $u^*=u$ q.e.\ in\/ $\Om$.
\end{thm}

\begin{proof}
By Theorem~4.1 (and Remark~4.2) in 
Harjulehto--Kinnunen--Lukkari~\cite{HaKiLu07},
$u^*=u$ a.e., and thus also $u^*$ is a supersolution.
By Theorem~6.1 in 
Harjulehto--H\"ast\"o--Koskenoja--Lukkari--Marola~\cite{HHKLM07},
$u^*$ is superharmonic (see Section~\ref{sect-superh} below for the definition of
superharmonic functions).
It then follows from Theorem~6.7 in Harjulehto--Latvala~\cite{HarLat08}, 
that $u^*$ is quasicontinuous.

Moreover, if $u$ is quasicontinuous, then $u^*= u$ q.e.\ in $\Om$,
by Lemma~\ref{lem-quasicont-ae-qe}.
\end{proof}

\medskip

\emph{In the rest of this section we assume that\/ $\Om$ is a bounded
nonempty open set.}

\medskip

\begin{thm}\label{thm-obst-prob-bdd}
If $\psi$ is bounded from above, $f$ is bounded, and $\K_{\psi,f}\ne \emptyset$,
then there exists a solution $u$ of the $\K_{\psi,f}$-obstacle problem,
and the solution is unique up to sets of measure zero. 
Moreover, $u^*$ is the unique lsc-regularized solution,
and $u^*$ is bounded.
\end{thm}

\begin{proof}
The existence is proved as in Appendix I in 
Heinonen--Kilpel\"ainen--Martio~\cite{HeKiMa}, namely by
showing the monotonicity, coercivity and weak continuity for the operator
\[
{\cal L}_\px:\{\grad v:v\in\K_{\psi,f}\}\to L^{p'(\cdot)}(\Om,\R^n),\quad \hbox{where }
1/p(x) +1/p'(x)\equiv 1,
\] 
defined by
\[
\langle {\cal L}_\px{\bf v},{\bf u} \rangle 
   := \int_\Om p(x) |{\bf v}(x)|^{p(x)-2} {\bf v}(x) \cdot  {\bf u}(x) \,dx.
\]
 These properties are for the variable exponent verified in the same way 
as in the constant exponent case, cf.\ Appendix I in~\cite{HeKiMa}
and p.~3427 in Harjulehto--H\"ast\"o--Koskenoja--Lukkari--Marola~\cite{HHKLM07}.

The uniqueness follows from Theorem~3.2 in~\cite{HHKLM07}.
Indeed, if $u$ and $v$ are solutions of the obstacle problem, 
then both are supersolutions and $\min\{u,v\}\in\K_{\psi,f}$.
Theorem~3.2 in~\cite{HHKLM07} then implies that $u\le v$ and $v\le u$
a.e.

As for the last part, $u^*=u$ a.e.\ by Theorem~\ref{thm-u*}, 
and thus $u^*$ is also a solution
of the $\K_{\psi,f}$-obstacle problem. Since $u^*$ is independent 
of which solution $u$ we choose of
the $\K_{\psi,f}$-obstacle problem, we see that
it is the unique lsc-regularized solution.

Let $M=\max\{\sup |f|,\sup \psi\}$.
Then the truncation $v:=\max\{\min\{u,M\},-M\}$ of $u$
at $\pm M$ is also a solution, and by the uniqueness we see
that $|u^*| \le M$. 
\end{proof}

\begin{thm}
Assume that $\psi: \Om \to [-\infty,\infty)$ is continuous\/ 
\textup{(}as an extended real-valued function\/\textup{)}
and bounded from above, 
that $f$ is bounded, and that $\K_{\psi,f}\ne \emptyset$. 
Then the lsc-regularized solution
of the $\K_{\psi,f}$-obstacle problem is continuous.
\end{thm}

\begin{proof}
See Theorem~4.11 in~\cite{HHKLM07}.
\end{proof}

\begin{remark} \label{rmk-soln-cont}
A direct consequence is that if $u$ is a 
locally bounded
solution, in the sense of Definition~\ref{deff-soln}, then $u^*$ is continuous.
Indeed, if $u$ is a solution, then it is locally a solution of an
unrestricted obstacle problem with itself as boundary values. 
Hence $u^*$ is locally continuous, i.e.\ continuous.
\end{remark}

\begin{deff}\label{Hf-defn}
Let $f \in \Wpx(\Om)$ be bounded.
Then we define the \emph{Sobolev solution} $\oHp f$ of the
Dirichlet problem with boundary values $f$ to be the continuous solution
of the $\K_{-\infty,f}$-obstacle problem.
\end{deff}

Note that $\oHp f$  depends also on $\px$. 
Since $u=\oHp f$ is a solution of the unrestricted obstacle problem, 
i.e.\ with obstacle $-\infty$, it follows 
that
\begin{equation}   \label{eq-obst=sol}
         \int_\Om  |\nabla u|^{\p}\, dx 
      \le \int_{\Om} |\nabla (u+\phi)|^{\p}\, dx.
\end{equation}
for all $\phi \in \Wpx_0(\Om)$ and in particular for all
$\phi \in C_0^\infty(\Om)$.
Subtracting
\[
\int_{A} |\nabla u|^{\p}\, dx = \int_{A} |\nabla (u+\phi)|^{\p}\, dx < \infty,
\]
where $A=\{x \in \Om : \phi(x)=0\}$,
from both sides of~\eqref{eq-obst=sol} 
shows that $u$ is a continuous solution 
in the sense of Definition~\ref{deff-soln}, i.e.\ 
a \px-harmonic function.

The following comparison principle will be important for us.

\begin{lem} \label{lem-comp-principle}
\textup{(Comparison principle)}
If $f_1,f_2 \in \Wpx(\Om)$ are bounded
and\/ $(f_1 -f_2)_\limplus \in \Wpx_0(\Om)$,
then $\oHp f_1 \le \oHp f_2$ in\/ $\Om$.
\end{lem}

It follows that if $f_1,f_2 \in \Lip(\overline{\Om})$ and
$f_1=f_2$ on $\bdy \Om$,
then $\oHp f_1 = \oHp f_2$.
We can therefore define $\oHp f$ for $f \in \Lip(\bdy \Om)$
to be $\oHp \ft$ for any extension $\ft \in \Lip(\overline{\Om})$
such that $\ft=f$ on $\bdy \Om$.
Among such extensions are the so-called McShane extensions,
see e.g.\ Theorem~6.2 in Heinonen~\cite{heinonen}.

In view of Lemma~\ref{lem-Wpx0-qe-on-bdry}, $(f_1-f_2)_\limplus\in\Wpx_0(\Om)$
whenever $f_1, f_2 \in\Wpx(\Om)$ are quasicontinuous in $\R^n$ and $f_1\le f_2$ 
q.e.\ on $\bdry\Om$.

The following generalization of the comparison principle above is sometimes
useful. 
Even though we will not use it in this paper,
we have chosen to include it here since the proof of it
is not more involved than a direct proof of Lemma~\ref{lem-comp-principle}.

\begin{lem} \label{lem-obst-le}
\textup{(Comparison principle for obstacle problems)}
Let $\psi_j \colon \Om \to [-\infty,\infty)$ be bounded
from above and $f_j \in \Wpx(\Om)$ be bounded and
such that $\K_{\psi_j,f_j} \ne \emptyset$. 
Let further $u_j$ be a
solution of the $\K_{\psi_j,f_j}$-obstacle problem, $j=1,2$.
If  $\psi_1 \le \psi_2$ a.e.\ in\/ $\Om$
and\/ $(f_1-f_2)_\limplus \in \Wpx_0(\Om)$,
then $u_1 \le u_2$ a.e.\ in\/ $\Om$.

Moreover, the lsc-regularizations satisfy $u^*_1\le u^*_2$
everywhere in\/ $\Om$.
\end{lem}

\begin{proof}[Proof of Lemma~\ref{lem-comp-principle}]
Let 
$\psi_1 = \psi_2 \equiv -\infty$.
After noting that $u_1^*= \oHp f_1$ and $u_2^*= \oHp f_2$ the result
follows from (the last part of) Lemma~\ref{lem-obst-le}.
\end{proof}

\begin{proof}[Proof of Lemma~\ref{lem-obst-le}]
Let $u=\min\{u_1,u_2\}$. Then
\begin{align*}
\Wpx_0(\Om) \ni u_1-f_1 \ge u-f_1 &= \min\{u_1-f_1,u_2-f_1\} \\
&\ge \min\{u_1-f_1,u_2-f_2 - (f_1-f_2)_\limplus\} \in\Wpx_0(\Om).
\end{align*}
Lemma~\ref{lem-police} implies that $u-f_1\in\Wpx_0(\Om)$.
As $u \ge \psi_1$ a.e.\ in $\Om$, we get that $u \in \K_{\psi_1,f_1}$.
Similarly $v=\max\{u_1,u_2\} \in \K_{\psi_2,f_2}$.

Let $A=\{x \in \Om : u_1(x) > u_2(x)\}$.
Since $u_2$ is a solution of the $\K_{\psi_2,f_2}$-obstacle problem,
we have that
\[
         \int_\Om |\grad u_2(x)|^{p(x)} \, dx
         \le  \int_\Om |\grad v(x)|^{p(x)} \, dx
         =  \int_A |\grad u_1(x)|^{p(x)} \, dx
         +  \int_{\Om \setm A} |\grad u_2(x)|^{p(x)} \, dx.
\]
Thus
\[
        \int_A |\grad u_2(x)|^{p(x)} \, dx
         \le  \int_A |\grad u_1(x)|^{p(x)} \, dx.
\]
It follows that
\begin{align*}
\int_\Om |\grad u(x)|^{p(x)} \, dx
        &=       \int_A |\grad u_2(x)|^{p(x)} \, dx
        +       \int_{\Om \setm A} |\grad u_1(x)|^{p(x)} \, dx 
\le       \int_\Om |\grad u_1(x)|^{p(x)} \, dx.
\end{align*}
As $u_1$ is a solution of the $\K_{\psi_1,f_1}$-obstacle problem, so is $u$.
By the uniqueness in Theorem~\ref{thm-obst-prob-bdd}, we have
\[   
    u_1=u=\min\{u_1,u_2\}
    \quad \text{a.e.\ in } \Om,
\] 
and thus $u_1 \le u_2$ a.e.\ in $\Om$.

The pointwise comparison of the lsc-regularizations follows directly
from their definitions and the above a.e.-inequality.
\end{proof}

\section{Superharmonic functions}
\label{sect-superh}

In this section we consider superharmonic functions and show that they
are lsc-regularized. This in turn leads to the characterization of bounded
superharmonic functions advertised in Theorem~\ref{thm-superh-char-bdd},
and to another characterization of general superharmonic functions.

\begin{deff} \label{def-superharm}
A function $u \colon \Om \to (-\infty,\infty]$ is
\emph{superharmonic}
in $\Om$ if
\begin{enumerate}
\renewcommand{\theenumi}{\textup{(\roman{enumi})}}%
\item \label{cond-a} $u$ is lower semicontinuous;
\item \label{cond-b}
 $u$ is finite almost everywhere;
\item \label{cond-c}
for every nonempty open set $\Om' \Subset \Om$
and all functions
$v \in C(\clOmprime)$ which are 
\px-harmonic in $\Om'$ 
and satisfy $v \le u$ on $\bdy \Om'$,
it is true that $v \le u$ in $\Om'$.
\end{enumerate}
A function $u \colon \Om \to [-\infty,\infty)$ is
\emph{subharmonic}
if $-u$ is superharmonic.
\end{deff}

In the variable exponent literature superharmonic functions are often
assumed to belong to $L^t\loc(\Om)$ for some $t>0$,
see e.g.\ Latvala--Lukkari--Toivanen~\cite{LLT}. 
For our purposes the more general definition above is sufficient.
In the constant $p$ case condition \ref{cond-b} is usually replaced
by the equivalent condition%
\medskip
\begin{enumerate}
\stepcounter{enumi}
\renewcommand{\theenumi}{\textup{(\roman{enumi}$'$)}}%
\item \label{cond-b'}
$u \not \equiv \infty$ in every
component of $\Om$. 
\end{enumerate}
\medskip
Whether this equivalence is true also for variable exponents is not known.
However, for the results in this paper we could as well have
replaced \ref{cond-b} by \ref{cond-b'}
and required that $u$ in Theorem~\ref{thm-superh-supermin-char}
satisfies \ref{cond-b'}.

The following lemma is well known and easily proved directly from the definition.

\begin{lem} \label{lem-superh-min}
If   $u$ and $v$ are superharmonic,
then so is\/ $\min\{u,v\}$.
\end{lem}

The following result is well known for constant $p$, but seems to be new in the variable exponent setting.

\begin{thm} \label{thm-superh-lsc-reg}
If a function is superharmonic, then it is lsc-regularized.
\end{thm}

\begin{proof}
Let $u$ be a superharmonic function and $x_0 \in \Om$ be arbitrary.
Since $u$ is lower semicontinuous, 
\[
     u(x_0) \le \liminf_{y \to x_0} u(y) \le \essliminf_{y \to x_0} u(y)=:u^*(x_0).
\]

In order to obtain the converse inequality we assume first that $u$ is bounded from above.
Without loss of generality we can assume that $u(x_0) >0$.
Let $0 < \de \le  u(x_0) $ be arbitrary.
By the lower semicontinuity of $u$, we can find a ball $B\ni x_0$
such that $2B\Subset\Om$ and $u\ge u(x_0)-\de$ in $2B$.
Then $v=u-(u(x_0)-\de)$ is a bounded nonnegative superharmonic function in $2B$.

Theorem~6.5 in Harjulehto--H\"ast\"o--Koskenoja--Lukkari--Marola~\cite{HHKLM07}
provides us with an increasing sequence of continuous supersolutions
$v_j$ in $B$ such that $v_j\nearrow v$ everywhere in $B$.
Theorem~3.7 and Remark~3.8 in Harjulehto--Kinnunen--Lukkari~\cite{HaKiLu07} 
imply the following weak Harnack inequality for sufficiently small $R>0$ and some $q>0$,
\begin{equation}   \label{eq-weak-Harnack}
\vint_{B(x_0,2R)} v_j^q\,dx  
         \le C \Bigl( \essinf_{B(x_0, R)} v_j + R \Bigr)^q,
\end{equation}
where the constants $q$ and $C$ depend on the bound for $v$, 
but not on $R$. Indeed, the proof of Lemma~3.6 in~\cite{HaKiLu07} reveals that
for a bounded $v$, the $L^s(B(x_0, 4R))$-norm in (3.33) in~\cite{HaKiLu07} 
can be substituted by the $L^{\infty}(\Om)$-norm, which gives the independence 
of $q$ on $R$.
We can clearly assume that $q<1$.
Since $v_j$ is continuous, the right-hand side in~\eqref{eq-weak-Harnack} 
is majorized by
\[
C(v_j(x_0)+R)^q \le C(v(x_0)+R)^q = C(\de+R)^q \le C(\de^q+R^q).
\]
Inserting this into~\eqref{eq-weak-Harnack}  and letting $j\to\infty$ 
gives
\[
C(\de^q+R^q) \ge \vint_{B(x_0,2R)} (u-(u(x_0)-\de))^q\,dx  
\ge \vint_{B(x_0,2R)} u^q\,dx  - (u(x_0)-\de)^q.
\] 
Hence
\[
(u(x_0)-\de)^q + C\de^q \ge \vint_{B(x_0, 2R)} u^q\,dx - CR^q
\ge \Bigl( \essinf_{B(x_0, 2R)}u \Bigr)^q - CR^q \to u^*(x_0)^q, 
\]
as $R\to0$.
Since $\de$ was arbitrary, we conclude that $u(x_0)\ge u^*(x_0)$
if $u$ is bounded from above.

Let us now consider the case when $u$ is unbounded.
Let $a < u^*(x_0)$ be real. 
Then $u_a:=\min\{u,a\}$ is superharmonic, by Lemma~\ref{lem-superh-min},
and thus $u_a$ is lsc-regularized by the first part of the proof.
Hence
\[
    u(x_0) \ge u_a(x_0) = \essliminf_{y \to x_0} u_a(y) 
    = \min\Bigl\{a, \essliminf_{y \to x_0} u(y)\Bigr\} 
    = \min \{a,u^*(x_0)\} = a.
\]
As $a$ was arbitrary we see that $u(x_0) \ge u^*(x_0)$.
\end{proof}

We are now ready to obtain the characterization of 
superharmonic functions in 
Theorem~\ref{thm-superh-char-bdd},
i.e.\  that a function  locally bounded from above is superharmonic
if and only if it is an lsc-regularized supersolution.

\begin{proof}[Proof of Theorem~\ref{thm-superh-char-bdd}.]
Assume first that $u$ is superharmonic.
Then $u$ is lsc-regu\-lar\-ized by Theorem~\ref{thm-superh-lsc-reg}.
That $u$ is locally bounded from below follows directly
from the lower semicontinuity (and the fact that $u$ does not take the
value $-\infty$).
Hence $u$ is locally bounded and Corollary~6.6 in 
Harjulehto--H\"ast\"o--Koskenoja--Lukkari--Marola~\cite{HHKLM07}
shows that $u$ a supersolution.

The converse follows directly from
Theorem~6.1 in \cite{HHKLM07}.
\end{proof}

For unbounded functions the characterization is 
(necessarily)
a bit more involved. 

\begin{thm} \label{thm-superh-supermin-char}
Let $u\colon \Om \to (-\infty,\infty]$ be a function
which is finite a.e.
Then the following are equivalent\/\textup{:}
\begin{enumerate}
\item \label{item-superh}
$u$ is superharmonic in\/ $\Om$\textup{;}
\item \label{item-superh-min}
$\min\{u,k\}$ is superharmonic in\/ $\Om$ for all $k=1,2,\ldots$\,\textup{;}
\item \label{item-lsc-reg}
$u$ is lsc-regularized\textup{,} and\/
$\min\{u,k\}$ is a supersolution in\/ $\Om$ for all $k=1,2,\ldots$\,\textup{;}
\item \label{item-lsc-reg-min}
$\min\{u,k\}$ is an lsc-regularized 
supersolution in\/ $\Om$ for all $k=1,2,\ldots$\,\textup{.}
\end{enumerate}
\end{thm}

\begin{proof}
\ref{item-superh} $\imp$ \ref{item-superh-min}
This follows from Lemma~\ref{lem-superh-min}.

\ref{item-superh-min} $\imp$ \ref{item-superh} 
That $u$ is lower semicontinuous follows directly from 
the fact that $\min\{u,k\}$, $k=1,2,\ldots$, are lower semicontinuous.
Let next $\Om' \Subset \Om$ be a nonempty open set 
and 
$v \in C(\clOmprime)$ be \px-harmonic
in $\Om'$ 
satisfying  $v \le u$ on $\bdy \Om'$.
Let $m=\sup_{\clOmprime} v < \infty$ and let $k>m$ be a positive integer.
Then $v \le \min\{u,k\}$ on $\bdy \Om'$.
Since $\min\{u,k\}$ is superharmonic it follows that
$v \le \min\{u,k\} \le u$ in $\Om'$.
Thus $u$ is superharmonic.

\ref{item-superh-min} $\eqv$ \ref{item-lsc-reg-min} 
This follows from Theorem~\ref{thm-superh-char-bdd}.

\ref{item-superh} $\imp$ \ref{item-lsc-reg}
That $u$ is lsc-regularized follows from Theorem~\ref{thm-superh-lsc-reg}.
That $\min\{u,k\}$ is a supersolution follows from 
the already shown implication
\ref{item-superh} $\imp$ \ref{item-lsc-reg-min}.

\ref{item-lsc-reg} $\imp$ \ref{item-lsc-reg-min} 
It is enough to show that $\min\{u,k\}$ is lsc-regularized, but this
follows directly from the fact that $u$ is lsc-regularized.
\end{proof}

\section{The Kellogg property}
\label{sect-cont-bdy-values}

\emph{From now on we assume that\/ $\Om$ is a bounded nonempty
open set.}

\medskip

In this section we extend the definition of Sobolev solutions of the
Dirichlet problem (Definition~\ref{Hf-defn})
to continuous boundary data and show that 
the solutions are
\px-harmonic.
We also introduce regular and irregular boundary points and prove the
Kellogg property.

\begin{deff} \label{def Hp cont}
Given $f \in C(\bdy \Om)$, define $\oHp f \colon \Om \to \R$ by
\[
\oHp f (x) 
       = \sup_{\Lip(\bdy \Om) \ni \phi \le f} 
                                              \oHp \phi(x),
       \quad x \in \Om.
\] 
\end{deff}

Here we abuse notation, since if 
$f \in \Wpx(\Om)$, then $\oHp f$ has already been defined by Definition~\ref{Hf-defn}.
However, as continuous functions can be uniformly
approximated by Lipschitz functions, 
the comparison
principle 
(Lemma~\ref{lem-comp-principle}),
together with the fact that $\oHp (f+a) = \oHp f +a$ for $a \in \R$, 
shows that the two definitions of $\oHp f$ coincide
in this case.

The comparison principle (Lemma~\ref{lem-comp-principle}) extends immediately  
to functions in $C(\bdy \Om)$
in the following way.

\begin{lem} \label{lem-comp-principle-cont}
\textup{(Comparison principle)}
If $f_1,f_2\in C(\bdy \Om)$ and 
$f_1\le f_2$ q.e.\ on $\bdy \Om$, 
then $\oHp f_1\le \oHp f_2$ in\/ $\Om$.
\end{lem}

Let us next show that $\oHp f$ is indeed \px-harmonic even for 
$f\in C(\bdry\Om)$.

\begin{lem} \label{lem-Hp-def}
Let $f \in C(\bdy \Om)$.
Then 
$\oHp f$ is \px-harmonic in\/ $\Om$
and
\[
  \oHp f (x)= \inf_{\Lip(\bdy \Om) \ni \phi  \ge f} 
                                 \oHp \phi(x)
             = \lim_{j\to\infty} \oHp f_j (x),
                                \quad x \in\Om,
\]
for every sequence\/ $\{f_j\}_{j=1}^\infty$
of functions in\/ $\Lip(\bdry \Om)$ converging uniformly
to $f$.
\end{lem}

\begin{proof}
Let $f_j \in \Lip(\bdy \Om)$ be such that 
$\sup_{\bdy \Om} |f - f_{j}| < 1/j$, $j=1,2,\ldots$\,.
Then $\sup_{\bdry \Om} |f_{j'} - f_{j''}| \le 2/j$
whenever $j', j'' \ge j$, and the comparison principle implies that for
all $x\in \Om$,
$$
Hf_{j'} (x) - \frac{2}{j} \le Hf_{j''} (x) 
\le Hf_{j'} (x) + \frac{2}{j},
$$
i.e.\ the sequence
$\{Hf_j (x)\}_{j=1}^\infty$ is a Cauchy sequence.
Hence, the limit
$h(x):=\lim_{j \to \infty} Hf_j(x)$ 
exists, and is a \px-harmonic function in $\Om$,
by the uniform convergence result in Corollary~5.3 in
Harjulehto--H\"ast\"o--Koskenoja--Lukkari--Marola~\cite{HHKLM07}.
Using the comparison principle again, it follows that
\begin{align*}
h(x) &= \lim_{j \to \infty} H(f_j - 1/j)(x)
       \le \sup_{\Lip(\bdy \Om) \ni \phi \le f} H\phi(x) \\
       &\le \inf_{\Lip(\bdy \Om) \ni \phi \ge f} H\phi(x)
       \le \lim_{j \to \infty} H(f_j + 1/j)(x)
       = h(x).
       \qedhere
\end{align*}       
\end{proof}

\begin{deff}\label{def-reg-pt}
Let $x_0 \in \bdy \Om$.
Then $x_0$ is \emph{regular} if
\[ 
         \lim_{\Om \ni y \to x_0} \oHp f(y)=f(x_0)
         \quad \text{for all } f\in C(\bdry\Om). 
\] 
We also say that $x_0$ is \emph{irregular} if it is not regular.
\end{deff}

See Theorem~\ref{reg-thm-1} below for characterizations
of regular boundary points.

Next we establish the  Kellogg property (Theorem~\ref{thm-kellogg}), 
which says that q.e.\ boundary point is regular.
The proof  
is based on the following 
pasting lemma, which may be of independent interest.

\begin{lem}    \label{lem-qsmin}
Let $x\in \partial \Om $ and $B=B(x,r)$.
Let  $f \in \Lip(\bdy \Om)$
be such that $f=M$
on $B \cap \bdy \Om$, where $M:=\sup_{\bdy \Om} f$.
Let further
\[
      u= \begin{cases}
             \oHp f & \text{in\/ $\Om$}, \\
             M &   \text{in $B \setm \Om$}.
         \end{cases}
\]
Then $u$ is a quasicontinuous supersolution
in $B$. 
\end{lem}

\begin{proof}
Extend $f$ to a Lipschitz function on $\overline{\Om}$
and let $f=M$ on $B \setm \overline{\Om}$.
Then $f \in \Lip(B) \subset \Wpx(B)$.
Let 
\[
    v=\begin{cases} 
          u-f & \text{in } B \cup \Om, \\
          0 & \text{otherwise}.
      \end{cases}
\]
Then $v=0$ in $\R^n \setm \Om$ and $v=\oHp f - f \in \Wpx_0(\Om)$.
As $v$ is continuous in $\Om$, Proposition~\ref{prop-Wpx0}
shows that $v \in \Wpx(B)$ and that $v$ is quasicontinuous.
Thus $u \in \Wpx(B)$ and $u$ is quasicontinuous in $B$.
By the comparison principle (Lemma~\ref{lem-comp-principle-cont}),
$u \le M$ in $B$.

To show that $u$ is a supersolution in $B$,
let $\phi \in C_0^\infty(B)$ be nonnegative.
We shall prove the inequality 
\begin{equation*} 
  \int _{\phi \ne 0} |\nabla u|^{\p}\,dx 
  \le \int _{\phi \ne 0} |\nabla (u+\phi)|^{\p} \,dx.
\end{equation*}

Let $\phi':=\min \{ \phi ,M-u\} \in \Wpx_0(B)$, which
is quasicontinuous and nonnegative in $B$.
Then $\phi' =0$  in $B \setm \Om$
and hence $\phi' \in \Wpx_0(B\cap \Om )$,
by Proposition~\ref{prop-Wpx0}.
Since $u$ is \px-harmonic in $B\cap \Om $, we have that
\begin{align*}
  \int _{\phi' \ne 0} |\nabla u|^{\p}\,dx 
    \le \int _{\phi' \ne 0} |\nabla (u+\phi')|^{\p} \,dx.
\end{align*}
Note that $\phi'=0\ne\phi$ if and only if $u=M$, in which case
$\grad u=0$ a.e.
Thus
\begin{align*}
  \int _{\phi \ne 0} |\nabla u|^{\p}\,dx 
   &=  \int _{\phi' \ne 0} |\nabla u|^{\p}\,dx 
       \le  \int _{\phi' \ne 0} |\nabla (u+\phi')|^{\p} \,dx.
\end{align*}
As $u+\phi'=\min\{u+\phi,M\}$ we have 
$|\nabla (u+\phi')| \le |\nabla (u+\phi)|$.
Since $\phi\ne0$ whenever $\phi'\ne0$, this finishes the proof.
\end{proof}

\begin{proof}[Proof of Theorem~\ref{thm-kellogg}]
For each  $j=1,2,\ldots$\,, we can cover $\bdy \Om$ by a finite number of balls
$B_{j,k}= B(x_{j,k}, 1/j)$, $1 \le k \le N_j$.
Let $\phi_{j,k}$ be a Lipschitz function with support in $3B_{j,k}$
such that  $0 \le \phi_{j,k} \le 1$
and  $\phi_{j,k}=1$ on $2B_{j,k}$.
Let further $\phi_{j,k,q}=q \phi_{j,k}$ for $0 < q \in \Q$.
Consider the sets 
\[
I_{j,k,q} = \Bigl\{ x \in  \itoverline{B}_{j,k} \cap\bdy \Om : 
     \liminf _{\Om \ni y \to x} \oHp \phi_{j,k,q}(y) < \phi_{j,k,q} (x)=q\Bigr\}.
\]
Note that $I_{j,k,q}$ contains only irregular points.
Let further
\[
      u_{j,k,q}= \begin{cases}
             \oHp \phi_{j,k,q} & \text{in $\Om$}, \\
             q &   \text{in $2B_{j,k} \setm \Om$},
         \end{cases}
\]
which is a quasicontinuous supersolution in $2B_{j,k}$ by Lemma~\ref{lem-qsmin}.
As $u_{j,k,q}$ is continuous in $\Om$, we have $u^*_{j,k,q}=H\phi_{j,k,q}$
in $\Om$.
By Theorem~\ref{thm-u*}, $u^*_{j,k,q}=u_{j,k,q}$ q.e.\ in $2B_{j,k}$
and hence
\[
        q=u_{j,k,q}(x)=u^*_{j,k,q}(x) 
        = \liminf_{\Om \ni y \to x}u^*_{j,k,q}(y)
             =\liminf_{\Om \ni y \to x} \oHp\phi_{j,k,q}(y)
\]
for q.e.\ $x \in \itoverline{B}_{j,k} \cap \bdy \Om$.
Thus $C_{\px}(I_{j,k,q})=0$. 

Now consider a function $\phi \in C(\bdy \Om)$ and assume that 
we do not have 
\[
 \lim_{\Om \ni y \to x} \oHp \phi(y) = \phi (x)
\]
for some $x \in \bdy \Om$.
By considering $-\phi$ if necessary, and adding a constant,
we can assume that $ \phi \ge 0$ and that
$     \liminf_{\Om \ni y \to x} \oHp \phi(y) < \phi (x)$.

Since $\phi$ is continuous we can find a ball $B_{j,k}$ 
containing the point $x$ so that
\[
      M:= \inf_{3B_{j,k} \cap \bdy \Om} \phi 
    > \liminf_{\Om \ni y \to x} \oHp \phi(y) \ge 0.
\]
We can then also find a rational $q$ such that
$M >q> \liminf_{\Om \ni y \to x} \oHp \phi(y)$.

Thus, $\phi_{j,k,q} \le \phi$ on $\bdy \Om$, and hence, by 
the comparison principle (Lemma~\ref{lem-comp-principle-cont}),
\[
       \liminf_{\Om \ni y \to x} \oHp \phi_{j,k,q}(y)
       \le \liminf_{\Om \ni y \to x} \oHp \phi(y)
       < q = \phi_{j,k,q} (x),
\]
i.e.\ $x \in I_{j,k,q}$.
Thus
\begin{equation} \label{eq-Ip}
    I_p = \bigcup_{j=1}^\infty \bigcup_{k=1}^{N_j} 
          \bigcup_{\substack{q \in \Q \\ q>0}}     I_{j,k,q},
\end{equation}
is a countable union of sets
of zero \px-capacity, and hence itself of zero \px-capacity.
\end{proof}

\begin{remark}
It is easy to see that
\[
I_{j,k,q} = \bigcup_{l=1}^\infty 
      ( \itoverline{B}_{j,k}\cap \bdy \Om  \cap 
      \overline{\{y \in \Om: \oHp \phi_{j,k,q}(y) <q-1/l\}}),
\]
is a countable union of compact sets. 
Together with the identity \eqref{eq-Ip}
this shows that $I_{\px}$ is an $F_\sigma$ set.
\end{remark}

\section{Removable singularities}
\label{sect-remove}

In this section we are going to prove Theorem~\ref{thm-removability-qharm-intro}.
Let us first state it in a slightly more precise form.

\begin{thm} \label{thm-removability-qharm}
Let $F \subset \Om$ be relatively closed and such that $\Cpx(F)=0$.
Let $u$ be a bounded \px-harmonic function in\/ $\Om \setm F$.
Then $u$ has a unique \px-harmonic extension to\/ $\Om$ given by
\[
U(x)=  \essliminf_{\Om \setm F \ni y \to x} u(y), \quad x \in \Om.
\]
If moreover $u \in \Wpx(\Om \setm F)$, then
$U \in \Wpx(\Om)$ and\/ $\|U\|_{\Wpx(\Om)}=\|u\|_{\Wpx(\Om \setm F)}$.
\end{thm}

Note that the boundedness assumption cannot be omitted even in the
constant-$p$ case, as shown by the function $u(x)=-|x|^{(p-n)/(p-1)}$, which
is \pp-harmonic in $B(0,1)\setm\{0\}\subset \R^n$ but not in $B(0,1)$.
It also shows that
the assumption that $u$ be bounded from below cannot be dropped from 
Theorem~\ref{thm-removability-q} below either.

Theorem~\ref{thm-removability-qharm} follows directly from 
Proposition~\ref{prop-rem-imp} below and the following removability
result for bounded superharmonic functions.

\begin{thm} \label{thm-removability-q}
Let $F \subset \Om$ be relatively closed and such that $\Cpx(F)=0$.
Let $u$ be a  superharmonic function in\/ $\Om \setm F$ 
which is bounded from below.
Then $u$ has a unique superharmonic extension\/ $U$ to\/ $\Om$
given by 
\[
    U(x)=\essliminf_{\Om \setm F \ni y \to x} u(y),\quad x\in \Omega.
\]
If moreover $u \in \Wpx(\Om \setm F)$, then
$U \in \Wpx(\Om)$ and\/ $\|U\|_{\Wpx(\Om)}=\|u\|_{\Wpx(\Om \setm F)}$.
\end{thm}

To prove Theorem~\ref{thm-removability-q} we need the following lemma.
It is similar to Lemma~3.1 in Lukkari~\cite{lukkari09}, 
but since one also needs that $0\le\phi_j\le1$, we provide the short proof
and clarify this point.

\begin{lem}\label{lem-seq}
 Let $K$ be a compact set. 
If $\Cpx(K)=0$, then there exists a sequence\/ 
$\{\phi_j\}_{j=1}^{\infty}$ of $C^{\infty}(\R^n)$ 
functions with the following properties\/{\rm:}
\begin{enumerate}
\item $0\leq \phi_j \leq 1$ in\/ $\R^n$ and 
$\phi_j \equiv 0$ in a neighbourhood of $K$\textup{;}
\item \label{item-b}
$\lim_{j \to \infty}\int_{\Om} |\nabla \phi_j|^{\p}\,dx = 0$\textup{;}
\item \label{item-c}
$\lim_{j\to \infty}\phi_j=1$ and\/ $\lim_{j\to \infty}\nabla \phi_j=0$ 
a.e.\ in\/ $\R^n$.
\end{enumerate}
\end{lem}

\begin{proof}
By Lemma~10.1.9 in 
Diening--Harjulehto--H\"ast\"o--R\r{u}\v{z}i\v{c}ka~\cite{DHHR},
the infimum in the definition of $\Cpx(K)$ can be taken over all
nonnegative $u\in C^\infty(\R^n)$ such that $u\ge1$ in a neighbourhood of $K$.
In fact, it follows from the proof (which implicitly uses the standard mollification through Theorem~9.1.6 in~\cite{DHHR}) that one can also assume that $0 \le u\le1$.
Thus, there are $u_j\in C^\infty(\R^n)$ such that $0\le u_j\le1$ in $\R^n$,
$u=1$ in a neighbourhood of $K$ and
\[
\int_{\R^n} (u_j^{p(x)} + |\grad u_j|^{p(x)})\,dx \to0,
\quad \text{as }j\to\infty.
\]
Letting $\phi_j=1-u_j$ and passing to a subsequence then finishes the proof.
\end{proof}

In what follows the Lebesgue measure of a set in $\R^n$ is denoted by $|\cdot|$.

\begin{proof}[Proof of Theorem~\ref{thm-removability-q}]
We first show the uniqueness. 
Let $V$ be any superharmonic extension 
of $u$. 
Since $V$ is lsc-regularized, by Theorem~\ref{thm-superh-lsc-reg},
and $|F|=0$, we see that
\[
    V(x)=\essliminf_{\Om \ni y \to x} V(y)
        =\essliminf_{\Om \setm F \ni y \to x} u(y)=U(x),
    \quad x\in \Omega,
\]
which shows the uniqueness.

Let us now turn to the existence.
Assume to begin with that $u$ is bounded.
By Theorem~\ref{thm-superh-char-bdd}, $u$ 
is an lsc-regularized supersolution in $\Om \setm F$.
It is straightforward that 
$U$ is bounded and lsc-regularized in $\Om$ and that $U=u$ in $\Omega\setm F$. 
We shall show that $U$ is a supersolution in $\Om$, 
and thus a bounded superharmonic extension of $u$, 
by Theorem~\ref{thm-superh-char-bdd} again, as required.

First, we show that $U \in W^{1, \px}\loc(\Om)$. 
Let $B\Subset \Om$ be a ball and $\eta\in C^{\infty}_{0}(B)$ 
be such that $0\leq \eta\leq 1$ and $\eta=1$ in $\tfrac12 B$.
Let $\{\phi_j\}_{j=1}^\infty$ be  
as in Lemma~\ref{lem-seq}, with $K=F \cap \supp \eta$, 
and consider $\eta_j=\eta\phi_j$. 
Since $u$ is bounded, we may assume that $u\leq 0$. 
Then $-u\eta_j^{\pplus}\in\Wpx_0(\Om \setminus F)$ is nonnegative and
compactly supported in $\Om \setminus F$. 
Thus we have
\[
\int_{\Om} \p|\nabla u|^{p(x)-2}\nabla u \cdot (-\eta_j^{\pplus}\nabla u
            -\pplus u\eta_j^{\pplus-1}\nabla \eta_j) \,dx\geq 0.
\]
Hence,
\begin{equation}  \label{eq-test-Om-F}
 \int_{\Om} \p|\nabla u|^{p(x)}\eta_j^{\pplus}\,dx\leq 
       \pplus\int_{\Om} \p|\nabla u|^{p(x)-1}|u| 
              \eta_j^{\pplus-1}  
         |\nabla \eta_j| \,dx.
\end{equation}
The last integrand can be estimated for every $0<\varepsilon<1$ and $x\in\Om$
using the Young inequality as
\begin{equation}  \label{eq-Young-at-x} 
\frac{|u|\,|\grad\eta_j|}{\eps} (\eps |\nabla u|^{p(x)-1} \eta_j^{\pplus-1}) 
   \le \frac{(|u|\,|\grad\eta_j|)^{p(x)}}{p(x)\,\eps^{p(x)}}
+ \frac{\eps^{p'(x)}}{p'(x)}  |\grad u|^{p(x)} \eta_j^{(\pplus-1)p'(x)}.
\end{equation}
Since $p'(x)\ge(\pplus)'=\pplus/(\pplus-1)$ and $1/p'(x) <1$, inserting this 
into~\eqref{eq-test-Om-F} yields
\begin{align}   \label{thm-removability-CI}
&\int_\Om p(x) |\nabla u|^{p(x)} \eta_j^{\pplus}\,dx \nonumber\\
  &\quad 
\le \frac{\pplus}{\eps^{\pplus}} \int_\Om |u|^{p(x)} \, |\grad\eta_j|^{p(x)}\,dx
+ \pplus \eps^{(\pplus)'} \int_{\Om} \p |\nabla u|^{p(x)} \eta_j^{\pplus} \,dx.
\end{align}
By choosing $\eps$ small enough we can include the last integral 
in the left-hand side. 
(Note that it is finite.)
As a consequence, we have for every $j=1,2,\ldots$,  
\begin{align} \label{eq-removability-est}
\int_{\tfrac12B} |\nabla u|^{p(x)} \phi_j^\pplus \,dx&\le \int_{\Om} |\nabla u|^{p(x)}{\eta_j}^{\pplus}\,dx 
\le C(\pplus) \int_{\Om} |u|^{p(x)} |\nabla \eta_j|^{p(x)} \,dx \nonumber\\
&\le C(\pplus,u) \int_{\Om} (|\nabla \phi_j| + |\nabla \eta|)^{p(x)} \,dx,
\end{align} 
since $u$ is bounded. 
By Lemma~\ref{lem-seq}\,\ref{item-b}, the last integral remains
bounded as $j\to\infty$.
Thus  we get from Lemma~\ref{lem-seq}\,\ref{item-c} and dominated convergence that
$\nabla u \in L^{\px}\bigl( \tfrac12 B \bigr)$.
Since $B\Subset\Om$ was arbitrary, $\nabla u \in L^{\px}\loc(\Om)$.

To conclude that $U \in W^{1, \px}\loc(\Om)$ it remains
to show that $\nabla u$ is the distributional gradient of $U$ in $\Om$.
To this end, let $\eta\in C^\infty_0(\Om)$ be arbitrary and let $\{\phi_j\}_{j=1}^\infty$ be 
as in Lemma~\ref{lem-seq} with $K=F \cap \supp \eta$. 
Then $\eta \phi_j\in C_{0}^{\infty}(\Om\setminus F)$.
Since $\grad u$ is the distributional gradient of $u$ in $\Om\setm F$,
we have
\[
0=\int_{\Om\setm F} (u\grad(\eta\phi_j) + \eta\phi_j\grad u)\,dx
= \int_\Om u\eta\grad\phi_j\,dx + \int_\Om \phi_j (U\grad\eta + \eta\grad u)\,dx.
\]
The first integral in the right-hand side tends to zero by Lemma~\ref{lem-seq}\,\ref{item-b}, (\ref{mod-norm-zero-conv}) and the H\"older inequality.
Since $0\le\phi_j\le1$ and $|U\grad\eta+\eta\grad u|\in L^1(\Om)$,
the last integral tends to 
\[
\int_\Om (U\grad\eta + \eta\grad u)\,dx
\]
by Lemma~\ref{lem-seq}\,\ref{item-c} and dominated convergence.
Thus, $\grad u$ is the distributional gradient of $U$ in $\Om$,
and $U \in W^{1, \px}\loc(\Om)$.

It remains to be proven that $U$ is a supersolution 
in the whole of $\Om$.
Let $0\le\eta\in C^{\infty}_{0}(\Om)$ be arbitrary.
As above,  $\eta \phi_j\in C_{0}^{\infty}(\Om\setminus F)$ 
is an admissible test function,
where $\{\phi_j\}_{j=1}^\infty$ again are given 
by Lemma~\ref{lem-seq} with $K=F \cap \supp \eta$. 
Since $u$ is a supersolution in $\Om\setminus F$, it holds that
\begin{equation}\label{thm-removability-eq}
\int_{\Om}\p|\nabla u|^{p(x)-2} (\nabla u \cdot \nabla \phi_j) \eta\, dx 
+ \int_{\Om}\p|\nabla u|^{p(x)-2} (\nabla u\cdot \nabla \eta) \phi_j\,dx
\ge 0.
\end{equation}
As $|\nabla u|^{\px-1}\in L\loc^{p'(\cdot)}(\Om)$, 
the H\"older inequality 
\eqref{Holder-ineq} implies that the first term in~\eqref{thm-removability-eq}
is majorized by
\begin{align*}
2 \pplus \max_{\Om}|\eta| \, \bigl\|\,|\grad u|^{\px-1}
     \bigr\|_{L^{p'(\cdot)}(\spt\eta)}
   \|\grad \phi_j\|_{L^{p(\cdot)}(\spt\eta)},
\end{align*}
which tends to zero as $j\to\infty$, by Lemma~\ref{lem-seq}\,\ref{item-b} 
together with~\eqref{mod-norm-zero-conv}.

As for the second term in~\eqref{thm-removability-eq},
the Young inequality shows that $ |\nabla u|^{\p-1} \in L^1(\supp \eta)$.
Hence the second term in~\eqref{thm-removability-eq} converges
by dominated convergence.
Letting $j\to \infty$ in \eqref{thm-removability-eq} 
then shows that
\begin{equation}
 \int_{\Om}\p|\nabla u|^{p(x)-2} \nabla u\cdot \nabla \eta\, dx \geq 0.
 \end{equation}
Thus $U$ is a supersolution in $\Om$.

Finally, consider the case when $u$ is unbounded.
By Lemma~\ref{lem-superh-min}, $u_k:=\min\{u,k\}$ is a bounded superharmonic 
function in $\Om \setm F$ which, by the above, has $U_k:=\min\{U,k\}$ as 
a bounded superharmonic extension to $\Om$.
By Theorem~\ref{thm-superh-supermin-char}, $U$ is superharmonic in $\Om$.

If moreover, $u \in \Wpx(\Om \setm F)$, then $\{u_k\}_{k=1}^\infty$
is a Cauchy sequence in $\Wpx(\Om \setm F)$.
By the above,  $\nabla u_k$ is the distributional gradient of $U_k$ in $\Om$.
Since $|F|=0$ it follows that 
$\|U_k\|_{\Wpx(\Om)}=\|u_k\|_{\Wpx(\Om \setm F)}$.
Hence $\{U_k\}_{k=1}^\infty$
is a Cauchy sequence in $\Wpx(\Om)$ with limit $U$,
and thus $\|U\|_{\Wpx(\Om)}=\|u\|_{\Wpx(\Om \setm F)}$.
\end{proof}

\begin{prop} \label{prop-rem-imp}
Assume that $F \subset \Om$ is relatively closed and\/ $|F|=0$.
Let $u$ be a bounded \px-harmonic function in\/ $\Om \setm F$,
which has a superharmonic extension $U$ and a
subharmonic extension $V$ to\/ $\Om$.
Then both $U$ and $V$ are unique and
$U=V$ is \px-harmonic in\/ $\Om$.
\end{prop}

\begin{proof}
Since $U$ is lsc-regularized and $|F|=0$, we have that
\[
 U(x)=\essliminf_{\Om \in y \to x} U(y)=\essliminf_{\Om \setm F \ni y \to x} u(y),\quad x\in \Om,
\]
and thus $U$ is unique.
Moreover, $U$ is bounded, as $u$ is bounded. 
By Theorem~\ref{thm-superh-char-bdd},
 $U$ is an lsc-regularized supersolution and since
$U=V$ a.e., $U$ is also a subsolution. 
Thus, $U$ is a solution.
Since $U$ is lsc-regularized, it follows from Remark~\ref{rmk-soln-cont} that
$U$ is continuous in $\Om$, and thus \px-harmonic in $\Om$.
Similarly $V$ is continuous in $\Om$,
and as $U=V$ a.e.\ in $\Om$ it follows that $U=V$ everywhere in $\Om$.
\end{proof}

The following two lemmas will be needed in the next section to prove
the trichotomy (Theorem~\ref{thm-trichotomy}). 
We state them already here to avoid a digression later on.

\begin{lem} \label{lem-connected}
Assume that
$G \subset \R^n$ is open and connected. 
If $F \subset G$ is
relatively closed with $\Cpx(F)=0$, then
$G \setm F$ is connected.
\end{lem}

\begin{proof}
Proposition~10.1.10 in 
Diening--Harjulehto--H\"ast\"o--R\r u\v zi\v cka~\cite{DHHR}
gives us that $C_{\pminus}(F)=0$. 
A simple modification of Lemma~2.46 in 
Heinonen--Kilpel\"ainen--Martio~\cite{HeKiMa} 
implies that $G \setm F$ is connected.
\end{proof}

\begin{lem} \label{lem-Cp-zero}
Assume that
$G \subset \R^n$ is open and connected. 
If $F \subsetneq G$ is relatively closed,
then $\Cpx(F)=0$ if and only if $\Cpx(\bdy F \cap G)=0$.
\end{lem}

\begin{proof}
The necessity follows immediately from $\Cpx(\bdy F \cap G)\leq \Cpx(F)=0$.

In order to show the converse implication, assume that $\Cpx(\bdy F\cap G)=0$. 
Then, by Lemma~\ref{lem-connected},
$G\setm (\bdy F\cap G)$ is connected and so ${\rm int}\, F=\emptyset$. 
Hence, $F= \bdy F \cap G$, and thus $\Cpx(F)=0$.
\end{proof}

In the setting of metric spaces Lemma~\ref{lem-Cp-zero} can be found as 
Lemma~4.5 in Bj\"orn--Bj\"orn~\cite{BBbook} for the constant $p$ case. 
Therein, the use of
Newtonian spaces simplifies the argument.

\section{Boundary regularity and trichotomy}
\label{sect-bdry-reg}

In this section we prove one of the main results of this paper, namely the 
trichotomy (Theorem~\ref{thm-trichotomy})
between regular, semiregular and strongly irregular boundary points.

Recall that an \emph{irregular} boundary point $x_0 \in \bdy \Om$ is
\emph{semiregular} if the limit
\begin{equation} \label{eq-semi-7}
     \lim_{\Om \ni y \to x_0} \oHp f(y)
        \quad \text{exists for all }f \in C(\bdy \Om);
\end{equation}
and \emph{strongly irregular} if 
for all $f \in C(\bdy \Om)$ there is
a sequence $\{y_j\}_{j=1}^\infty$ such that
\begin{equation} \label{eq-strong-7}
 \Om \ni y_j \to x_0 \text{ and } \oHp f(y_j) \to f(x_0),
 \quad \text{as } j \to \infty.
\end{equation}

\begin{proof}[Proof of Theorem~\ref{thm-trichotomy}]
\emph{Case} 1.
\emph{There is $r>0$ such that $\Cpx(B \cap \bdy \Om)=0$,
where $B=B(x_0,r)$.}

By Lemma~\ref{lem-Cp-zero}, $\Cpx(B\setm \Om)=0$ 
and thus $B \subset \overline{\Om}$.
Let $f \in C(\bdy \Om)$.
By Theorem~\ref{thm-removability-qharm}, the \px-harmonic function
 $\oHp f$ 
has a \px-harmonic extension $U$ to $\Om \cup B$.
Since $U$ is continuous we have
\[
     \lim_{\Om \ni y \to x_0} \oHp f(y) = U(x_0),
\] 
i.e.\ \eqref{eq-semi-7} holds and $x_0$ is either regular or semiregular.

\emph{Case} 2.
\emph{The capacity $\Cpx(B(x_0,r) \cap \bdy \Om)>0$ for all $r>0$.}
(Note that this is complementary to case~1.)

For every $j=1,2,\ldots$\,, we thus have
$\Cpx(B(x_0,1/j) \cap \bdy \Om)>0$,
and by the Kellogg property (Theorem~\ref{thm-kellogg})
there is a regular boundary point $x_j \in B(x_0,1/j) \cap \bdy \Om$.
(We do not require the $x_j$ to be distinct.)

As $x_j$ is regular, we can find $y_j \in B(x_j,1/j) \cap \Om$
so that $|\oHp f(y_j)-f(x_j)|<1/j$.
It follows directly that $y_j \to x_0$ and $\oHp f(y_j) \to f(x_0)$, 
as $j \to \infty$,
i.e.\ \eqref{eq-strong-7} holds, and thus $x_0$ is either
regular 
or strongly irregular.
\end{proof}

We finish this section by characterizing regular boundary points in several
ways.
Semiregular boundary points will be characterized in Section~\ref{sect-semi}.
In view of the trichotomy result this indirectly
characterizes the strongly irregular points as well.

\begin{thm} \label{reg-thm-1}
Let $x_0 \in \bdy \Om$
and $d(x):= d(x,x_0)$.
Then the following are equivalent\/\textup{:}
\begin{enumerate}
\item \label{reg-1} 
The point $x_0$ is a regular boundary point.
\item \label{perron-dist-1}
It is true that
\[
       \lim_{\Om \ni y \to x_0} \oHp (jd)(y) = 0
       \quad \text{for all }j=1,2,\ldots.
\]
\item \label{cont-x0-perron-1}
It is true that 
\[
        \lim_{\Om \ni y \to x_0}  \oHp f(y) = f(x_0)
\]
for all bounded $f \in\Wpx(\Om)$
such that $f(x_0):=\lim_{\Om \ni y \to x_0} f(y)$ exists.
\item \label{semicont-x0-perron-1}
It is true that 
\[
        \limsup_{\Om \ni y \to x_0}  \oHp f(y) \le \limsup_{\Om \ni y \to x_0} f(y)
\]
for all bounded $f \in\Wpx(\Om)$.
\end{enumerate}
\end{thm}

In the constant $p$ case it is enough if \ref{perron-dist-1} holds
for $j=1$, which is easily seen since $\oHp(jd)=j \oHp d$ in this case.
In the variable exponent case this latter fact is not true,
but it is not known whether it suffices that \ref{perron-dist-1} holds
for $j=1$ also in this case.
In fact, the situation is similar for \pp-parabolic equations in the sense
that if $u$ is a \pp-parabolic function and $a \in \R$, then $u+a$ is
\pp-parabolic, but $au$ is in general not \pp-parabolic.
In the \pp-parabolic case a similar characterization of boundary
regularity to the one above was obtained
by Bj\"orn--Bj\"orn--Gianazza--Parviainen~\cite{BBGP}.
Therein a characterization of boundary regularity in terms
of the existence of a family of barriers was also obtained.
It would be interesting to obtain a similar characterization
in our variable exponent elliptic case.
Whether one barrier could suffice for boundary regularity 
in the \pp-parabolic case or in the variable exponent elliptic case
is an open question.

\begin{proof}
\ref{reg-1} $\imp$ \ref{perron-dist-1} 
This follows directly from Definition~\ref{def-reg-pt} 
by taking $f=jd$ for $j=1,2,\ldots$\,.

\ref{perron-dist-1} $\imp$ \ref{semicont-x0-perron-1} 
Let $A >\limsup_{\Om \ni y \to x_0} f(y)$ be real and $M=\sup_{\bdry\Om} (f-A)_\limplus$.
Let further $r>0$ be such that $f(x) <A$ 
for $x \in B(x_0,r) \cap \bdy \Om$,
and let $j>M/r$ be an integer. 
Then $f \le A+Md/r < A+jd$ on $\bdy \Om$.
It follows from the comparison principle in Lemma~\ref{lem-comp-principle}
that
\[
     \limsup_{\Om \ni y \to x_0} \oHp f(y) 
      \le A+ \lim_{\Om \ni y \to x_0} \oHp (jd)(y) 
      = A.
\]
Letting $A \to \limsup_{\Om \ni y \to x_0} f(y)$ gives 
$\limsup_{\Om \ni y \to x_0} \oHp f(y) \le \limsup_{\Om \ni y \to x_0} f(y)$.

\ref{semicont-x0-perron-1} $\imp$ \ref{cont-x0-perron-1} 
Applying \ref{semicont-x0-perron-1} to $-f$ yields
\[
    \liminf_{\Om \ni y \to x_0} \oHp f(y) 
     = - \limsup_{\Om \ni y \to x_0} \oHp(-f)(y) 
     \ge - (-f(x_0)) = f(x_0).
\]
Together with \ref{semicont-x0-perron-1} this gives the 
desired conclusion.

\ref{cont-x0-perron-1} $\imp$ \ref{reg-1} 
Let $f \in C(\bdy \Om)$. By the comparison principle together 
with uniform approximation by Lipschitz functions we may 
as well assume that $f \in \Lip(\bdy \Om)$. 
We find an extension $\ft \in \Lip(\overline{\Om})$
such that $\ft=f$ on $\bdy \Om$ (e.g.\ a McShane extension).
Then by definition and \ref{cont-x0-perron-1},
\[
        \lim_{\Om \ni y \to x_0}  \oHp f(y) 
        =         \lim_{\Om \ni y \to x_0}  \oHp \ft(y) = f(x_0).
        \qedhere
\]
\end{proof}

\section{Characterizations of semiregular points}
\label{sect-semi}
Similarly to regular points, semiregular points can be characterized by 
a number of equivalent conditions.
This will be done in Theorem~\ref{thm-rem-irr-char}, but before that we
obtain the following characterizations of relatively open sets of semiregular
points.

\begin{thm} \label{thm-irr-char-V}
Let $V \subset \bdy \Om$ be relatively open.
Then the following are equivalent\/\textup{:}
\begin{enumerate}
\renewcommand{\theenumi}{\textup{(${\rm \alph{enumi}}'$)}}%
\item \label{V-semireg}
The set $V$ consists entirely of semiregular points.
\item \label{V-R}
The set $V$ does not contain any regular point.
\item \label{V-Cp-V-bdy}
It is true that $\Cpx(V)=0$.
\item \label{V-alt-def-irr}
The set\/ $\Om \cup V$ is open in\/ $\R^n$,
and every
bounded \px-harmonic function in\/ $\Om$ has a \px-harmonic extension
to\/ $\Om \cup V$.
\item \label{V-alt-def-irr-super}
The set\/ $\Om \cup V$ is open in\/ $\R^n$, $|V|=0$,
and every
bounded superharmonic function in\/ $\Om$ has a superharmonic extension
to\/ $\Om \cup V$.
\item \label{V-rem-motiv}
For every $f \in C(\bdy \Om)$, the 
\px-harmonic extension
$\oHp f$ depends only on $f|_{\bdy \Om \setm V}$
\textup{(}i.e.\ if $f,h \in C(\bdy \Om)$ and
$f=h$ on $\bdy \Om \setm V$, then 
$\oHp f \equiv \oHp h$\textup{)}.
\end{enumerate}
\end{thm}

Together with the implication \ref{semireg} $\imp$ \ref{Cp-V-bdy}
in Theorem~\ref{thm-rem-irr-char} this theorem shows that
the set $S$ of all semiregular boundary points 
can be characterized as 
the largest relatively open subset of $\bdy \Om$ having any 
of the properties above. 
Equivalently, it can be written e.g.\ as 
\begin{equation}    \label{eq-char-S-rel-open}
    S= \bigcup \{ V \subset \bdy \Om :
            \Cpx(V)=0 \text{ and $V$ is relatively open}\}.
\end{equation}

By \ref{V-alt-def-irr} we also see that $S$ is
contained in the interior of $\overline{\Om}$, 
i.e.\ $S \subset \bdy\Om \setm \bdy \overline{\Om}$.
Note however that it can happen that
$S \ne \bdy\Om \setm \bdy \overline{\Om}$,
as the following examples show.

\begin{example}
Let $\Cpx(\{x\})>0$ and $G \ni x$, where $G$ is a bounded open set,
and let $\Om :=G \setm \{x\}$.
Then $x$ is regular with respect to $\Om$, by
the Kellogg property,
but $x \in \bdy\Om \setm \bdy \overline{\Om}$,
cf.\ Example~\ref{ex-punctured-ball}.
\end{example}

\begin{example}
Let $n=2$ and $p \equiv 2$.
Let further
$\Om$ 
be the slit disc $B((0,0),1) \setm ((-1,0]\times \{0\})$.
It is well known  that $\Om$ is regular, and hence $S=\emptyset$.
However, $\bdy \Om \setm \bdy \overline{\Om} = (-1,0]\times\{0\}$.
\end{example}

The strong minimum principle says that 
\emph{if\/ $\Om$ is connected, 
$u$ is superharmonic in\/ $\Om$ and $u$ attains
its minimum in\/ $\Om$, then $u$ is constant in\/ $\Om$.}
The proof of the implication \ref{V-alt-def-irr} $\imp$ \ref{V-semireg}
is considerably easier 
 when the strong minimum principle is available, 
but it is not known if it holds 
in our generality. 
The strong minimum principle for the variable exponent case was
obtained by 
Fan--Zhao--Zhang~\cite{FZZ} under the assumption
that $p\in C^1(\overline{\Om})$. 
Theorem~5.3 in Harjulehto--H\"ast\"o--Latvala--Toivanen~\cite{HHLT} shows that
the strong minimum principle holds also under the weaker assumption that $p$ 
satisfies a  Dini-type condition, see (5.1) in~\cite{HHLT}.

\begin{proof}   [Proof of Theorem~\ref{thm-irr-char-V}]
\ref{V-semireg} $\imp$ \ref{V-R}
This is trivial.

\ref{V-R} $\imp$ \ref{V-Cp-V-bdy}
This follows directly from the Kellogg property (Theorem~\ref{thm-kellogg}).

\ref{V-Cp-V-bdy} $\imp$ \ref{V-alt-def-irr-super}
Let $x \in V$ and let $G$ be a connected neighbourhood of $x$, such that
that $G \cap \bdy \Om \subset V$.
By Lemma~\ref{lem-connected} sets of zero \px-capacity cannot separate space,
and hence $G \setm \bdry\Om$ must be connected.
Since $G\setm\bdry\Om=(G\cap\Om)\cup(G\setm\overline{\Om})$ and 
$G\cap\Om\ne\emptyset$, we get that $G\subset\overline{\Om}$.
As $G\cap\bdry\Om\subset V$, this implies that $G\subset\Om\cup V$. 
Since $x\in V$ was arbitrary, we conclude that $\Om\cup V$ is open. 
That $|V|=0$ follows directly from the fact that $\Cpx(V)=0$.
The extension is now provided by Theorem~\ref{thm-removability-q}.

\ref{V-alt-def-irr-super} $\imp$ \ref{V-alt-def-irr}
Let $u$ be a bounded \px-harmonic function on $\Om$.
By assumption, $u$ has a superharmonic extension $U$ to $\Om \cup V$.
Also $-u$ has a superharmonic extension $W$ to $\Om \cup V$.
Thus $-W$ is a subharmonic extension of $u$ to $\Om \cup V$.
By Proposition~\ref{prop-rem-imp},
$U=-W$ is \px-harmonic.

\ref{V-Cp-V-bdy} $\imp$ \ref{V-rem-motiv}
Let $f,h\in C(\bdry\Om)$ with $f=h$ on $\bdry\Om\setm V$.
Then $\eta:=h-f\in C(\bdry\Om)$ and $\eta=0$ on $\bdry\Om\setm V$.

Let $f_j\in\Lip_c(\R^n)\subset\Wpx(\R^n)$ converge uniformly to $f$
on $\bdry\Om$. (Here $\Lipc(\R^n)$ consists of Lipschitz
functions on $\R^n$ with compact support.)
Let also $\eta'_j\in\Lip_c(\R^n)$ be such that
$|\eta'_j-\eta|<1/j$ on $\bdry\Om$.
Letting $\eta_j=(\eta'_j-1/j)_\limplus-(\eta'_j+1/j)_\limminus$ we see that
$\eta_j=0$ on $\bdry\Om\setm V$ and $\eta_j\to\eta$ uniformly on $\bdry\Om$.

Since $\Cpx(V)=0$, Lemma~\ref{lem-Wpx0-qe-on-bdry} shows that 
$\eta_j\in\Wpx_0(\Om)$ and hence $\oHp f_j = \oHp(f_j+\eta_j)$.
Since $f_j+\eta_j\to f+\eta=h$ and $f_j\to f$ uniformly on $\bdry\Om$,
Lemma~\ref{lem-Hp-def} implies
$\oHp(f_j+\eta_j)\to \oHp h$ and $\oHp f_j\to \oHp f$ in $\Om$,
i.e.\ $\oHp h =\oHp f$.

\ref{V-rem-motiv} $\imp$ \ref{V-R}
Let $x_0\in V$.
As $V$ is relatively open in $\bdry\Om$, there exists $r>0$ such that
$B(x_0,r)\cap\bdry\Om \subset V$.
Define $f(x)=(1-d(x,x_0)/r)_\limplus$.
Then $f\in\Lip_c(\R^n)$ and $f=0$ on $\bdry\Om\setm V$.
Using~\ref{V-rem-motiv} we conclude that $\oHp f = \oHp 0 = 0$ in $\Om$.
Since $f(x_0)=1$, this shows that $x_0$ is not regular.

\ref{V-alt-def-irr} $\imp$ \ref{V-semireg}
Let $x_0 \in V$ and $f\in C(\bdy \Om)$.
Since $\oHp f$ has a \px-harmonic extension $u$ to $\Om \cup V$,
it follows that
\begin{equation}  \label{eq-lim-ex}
       \lim_{\Om \ni y \to x_0} \oHp f(y)=
       \lim_{\Om \ni y \to x_0} u(y)=u(x_0),
\end{equation}
and thus the limit in the left-hand side always exists.
It remains to show that $x_0$ is irregular.

As $V$ is relatively open in $\bdry\Om$, there exists $r>0$ such that
$B(x_0,r)\cap\bdry\Om \subset V$.
Define $h(x)=(1-d(x,x_0)/r)_\limplus$.
By assumption,  $\oHp  h$ has a \px-harmonic extension $U$ to $\Om\cup V$.
We shall show that $U\equiv0$ in $\Om\cup V$, as then
\[
      \lim_{\Om  \ni y \to x_0}\oHp h(y) = 
    \lim_{\Om  \ni y \to x_0} U(y) = 0 \ne 1=  h(x_0),
\]
i.e.\  $x_0$ is irregular.

As $0 \le h \le 1$,
we see that $0 \le \oHp h \le 1$ and also $0 \le U \le 1$.
By the  Kellogg property (Theorem~\ref{thm-kellogg}), 
\[
\lim_{\Om\ni y\to x} U(y) = \lim_{\Om\ni y\to x} \oHp h(y) = h(x)=0 
\]
q.e.\ in $\bdry(\Om\cup V)=\bdry\Om\setm V$, i.e.\ 
for all $x\in\bdry(\Om\cup V)\setm E$, where $\Cpx(E)=0$.
Moreover, \eqref{eq-lim-ex} applied to $y_j\in V$ instead of $x_0$ implies that
\[
0\le \limsup_{j\to\infty} U(y_j) \le \limsup_{\Om\ni y\to x} U(y) = 0,
\]
whenever $y_j\in V$ converge to some $x\in\bdry(\Om\cup V)\setm E$.
Hence,
\[
\lim_{\Om\cup V\ni y\to x} U(y) = h(x)=0 
\]
for $x\in\bdry(\Om\cup V)\setm E$.
Note that we cannot use the comparison principle 
(Lemma~\ref{lem-comp-principle}) 
directly to prove that $U\equiv0$ in $\Om\cup V$, since we do not know
that $U\in\Wpx_0(\Om \cup V)$.
(We know that $U\in\Wpx\loc(\Om \cup V)$ and that $U \in \Wpx(\Om)$,
but since it could a priori happen that $|V \setm \Om|> 0$,
we cannot, at this point, even deduce that $U \in \Wpx(\Om \cup V)$.)

Let $0 < \eps< 1$ and find an open set $G\supset E$ such that $\Cpx(G)<\eps$.
Let also $\phi\in\Wpx(\R^n)$ be such that $0\le\phi\le 1$, $\phi\equiv 1$ 
on $G$ and 
\begin{equation}  \label{eq-cap-zero-small-modular}
\int_{\R^n}(\phi^{p(x)} + |\grad\phi|^{p(x)})\,dx \le\eps.
\end{equation}
For every $x\in\bdry(\Om\cup V)\setm E$ there exists a ball $B_x\ni x$ such that
$0\le U<\eps$ in $2B_x\cap(\Om\cup V)$.
Exhaust $\Om\cup V$ by open sets
\[
\Om_1\subset\Om_2\subset\ldots \Subset\bigcup_{j=1}^\infty\Om_j = \Om\cup V.
\]
Then
\[
\overline{\Om\cup V} 
\subset \bigcup_{j=1}^\infty\Om_j \cup G\cup \bigcup_{x\in\bdry(\Om\cup V)\setm E} B_x.
\]
By compactness, there exists $j>1/\eps$ such that 
\[
\bdry\Om_j\subset G\cup \bigcup_{x\in\bdry(\Om\cup V)\setm E} B_x.
\]
Then $0\le U\le\eps$ on $\bdry\Om_j\setm G$ and
as $\phi\ge\chi_G$, we obtain $U\le\eps+\phi$ on $\bdry\Om_j$.
Since $U\in\Wpx(\Om_j)$, it is its own $\px$-harmonic extension in $\Om_j$,
i.e.\ $U=\oHp_{\Om_j}U$.
If we let $v=\oHp_{\Om_j}\phi$, then $U\le\eps+v$ in $\Om_j$,
by the comparison principle (Lemma~\ref{lem-comp-principle}).
The Poincar\'e inequality (Theorem~8.2.4 in 
Diening--Harjulehto--H\"ast{\"o}--R{\r u}\v{z}i\v{c}ka~\cite{DHHR}),
applied to $v-\phi\in\Wpx_0(\Om_j)$ and some ball $B\supset\Om_j$, yields
\begin{align}  \label{eq-use-PI}
\|v-\phi\|_{\Lpx(B)} 
&\le C_B \|\grad(v-\phi)\|_{\Lpx(B)} \\
&\le C_B ( \|\grad v\|_{\Lpx(B)} + \|\grad\phi\|_{\Lpx(B)} ).
\nonumber
\end{align}
Since $v=\oHp_{\Om_j}\phi$, we conclude from \eqref{norm-mod-le-1} and \eqref{eq-cap-zero-small-modular} that
\begin{align*}
 \|\grad v\|_{\Lpx(B)} &\le \biggl( \int_B |\grad v(x)|^{p(x)}\,dx \biggr)^{1/\pplus}
\le \biggl( \int_B |\grad \phi(x)|^{p(x)}\,dx \biggr)^{1/\pplus}.
\end{align*}
Inserting this into \eqref{eq-use-PI}, together 
with~\eqref{norm-mod-le-1} again and~\eqref{eq-cap-zero-small-modular}, gives
\begin{align*}
\|v\|_{\Lpx(\Om_j)} &\le \|\phi\|_{\Lpx(B)} + \|v-\phi\|_{\Lpx(B)} \\
&\le \biggl( \int_B |\phi(x)|^{p(x)}\,dx \biggr)^{1/\pplus}
+ 2C_B \biggl( \int_B |\grad\phi(x)|^{p(x)}\,dx \biggr)^{1/\pplus}
\le 3C_B \eps^{1/\pplus}.
\end{align*}
Here we assume that $C_B\ge1$. 
It follows that 
\[
\|U\|_{\Lpx(\Om_j)} \le \|\eps+v\|_{\Lpx(\Om_j)} \le \eps\|1\|_{\Lpx(\Om)} + 3C_B \eps^{1/\pplus}.
\]
Letting $\eps\to0$ (and thus $j\to\infty$) implies
$\|U\|_{\Lpx(\Om\cup V)} =0$, and hence $U\equiv0$ in $\Om \cup V$.
\end{proof}

We are now ready to characterize semiregular boundary points in several
different ways.
Note that \ref{semireg-local} below
shows that semiregularity is a local property, even though
we have not shown that regularity is a local property.
The latter however follows from the Wiener criterion,
whose usage we have avoided in this paper.
It thus also follows that strong irregularity is a local property.
It would be nice to have a simpler and more direct proof
(without appealing to the Wiener criterion)
that regularity is a local property.
Such proofs are available in the constant $p$ case, using
barrier characterizations,
see Theorem~9.8 and Proposition~9.9 in
Heinonen--Kilpel\"ainen--Martio~\cite{HeKiMa} for the weighted $\R^n$ case,
and Theorem~6.1 in Bj\"orn--Bj\"orn~\cite{BB} (or \cite[Theorem~11.11]{BBbook}) for 
metric spaces.

\begin{thm} \label{thm-rem-irr-char}
Let $x_0 \in \bdy \Om$, $\de >0$ and $d(y)=d(y,x_0)$.
Then the following are equivalent\/\textup{:}
\begin{enumerate}
\item \label{semireg}
The point $x_0$ is semiregular.
\item \label{semireg-local}
The point $x_0$ is semiregular with respect to $G:=\Om \cap B(x_0,\de)$.
\item \label{not-reg-one-seq}
There is no sequence\/
$\{y_j\}_{j=1}^\infty$ such that\/
$\Om \ni y_j \to x_0$, as $j \to \infty$,
and 
\[
 \lim_{j \to \infty} \oHp f(y_j) =f(x_0)
\quad \text{for all $f \in C(\bdy \Om)$}.
\]
\item \label{R}
It is true that $x_0 \notin \overline{\{x \in \bdy \Om : x \text{ is regular}\}}$.
\item \label{Cp-V-bdy}
There is a neighbourhood $V$ of $x_0$ 
such that $\Cpx(V \cap \bdy \Om)=0$.
\item \label{Cp-V}
There is a neighbourhood $V$ of $x_0$ 
such that $\Cpx(V \setm \Om)=0$.
\item \label{rem-irr}
The point $x_0$ is irregular and there is a neighbourhood $V$ of $x_0$ 
such that every
bounded \px-harmonic function in\/ $\Om$ has a \px-harmonic extension
to\/ $\Om \cup V$.
\item \label{alt-def-irr}
There is a neighbourhood $V$ of $x_0$ such that 
$V \subset \overline{\Om}$  and every
bounded \px-harmonic function in\/ $\Om$ has a \px-harmonic extension
to\/ $\Om \cup V$.
\item \label{alt-def-irr-super}
There is a neighbourhood $V$ of $x_0$ such that $|V \setm\nobreak \Om|=0$
and every bounded superharmonic function in\/ $\Om$ 
has a superharmonic extension to\/ $\Om \cup V$.
\item \label{rem-motiv}
There is a neighbourhood $V$ of $x_0$ such that
for every $f \in C(\bdy \Om)$, the \px-harmonic extension
$\oHp f$ depends only on $f|_{\bdy \Om \setm V}$
\textup{(}i.e.\ if $f,h \in C(\bdy \Om)$ and
$f=h$ on $\bdy \Om \setm V$, then 
$\oHp f \equiv \oHp h$\textup{)}.
\item \label{d-lim}
It is true that for some positive integer $j$,
\[
   \lim_{\Om \ni y \to x_0} \oHp (jd)(y) > 0.
\] 
\item \label{d-liminf}
It is true that for some positive integer $j$,
\[
   \liminf_{\Om \ni y \to x_0} \oHp (jd)(y) > 0.
\] 
\end{enumerate}
\end{thm}

\begin{remark} \label{rmk-strong}
Note that \ref{not-reg-one-seq}  says that if $x_0$ is strongly irregular,
then the sequence $\{y_j\}_{j=1}^\infty$ occurring in \eqref{eq-strong}
can be chosen independently of $f \in C(\bdy \Om)$.
\end{remark}

\begin{proof}
\ref{R} $\eqv$ \ref{Cp-V-bdy} 
 $\eqv$ \ref{rem-motiv} 
 $\imp$ \ref{semireg}
This follows directly from 
Theorem~\ref{thm-irr-char-V},
with $V$ in Theorem~\ref{thm-irr-char-V} corresponding to
$V \cap \bdy \Om$ here.

\ref{semireg} $\imp$ \ref{d-lim}
The limit
\[
       c_j:=\lim_{\Om \ni y \to x_0} \oHp (jd)(y)
\]
exists for $j=1,2,\ldots$.
If $c_j$ were $0$ for $j=1,2,\ldots$, then $x_0$ would be regular,
by 
Theorem~\ref{reg-thm-1}, 
a contradiction.
Thus $c_j>0$ for some positive integer $j$.

\ref{d-lim} $\imp$ \ref{d-liminf} 
$\imp$ \ref{not-reg-one-seq}
This is trivial.

$\neg$\ref{R} $\imp$ $\neg$\ref{not-reg-one-seq}
For each $j \ge 1$,
$B(x_0,1/j^2)\cap \bdy \Om$ contains a regular boundary point $x_j$.
Let $f_j=jd\in C(\bdry\Om)$. 
Then we can find $y_j \in B(x_j,1/j) \cap \Om$ so that
\[
        \frac{1}{j} > | f_j(x_j)-\oHp f_j(y_j)|. 
\]
Since $0\le f_j(x_j)\le1/j$, we have $\oHp f_j(y_j)\le 2/j$.
Moreover, $y_j \to x_0$ 
as $j \to \infty$.

Let now $f \in C(\bdy \Om)$.
Without loss of generality we may assume that $|f|\le M <\infty$
and that $f(x_0)=0$.
Let $\eps >0$. Then we can find $k$ such that
\[
      |f| \le \eps 
      \quad \text{on } B(x_0,1/k) \cap \bdy \Om.
\]
For $j\ge Mk$ we have $f_j\ge M$ on $\bdry\Om\setm B(x_0,1/k)$ and hence
$|f|\le\eps+ f_j$ on $\bdry\Om$.
It follows that for $j\ge Mk$,
\begin{alignat*}{2}
\oHp f(y_j) &\le \eps + \oHp f_j(y_j) \le \eps + \frac2j \to \eps ,
   &\quad& \text{as } j \to \infty
\intertext{and}
\oHp f(y_j) &\ge -\eps - \oHp f_j(y_j) \ge -\eps - \frac2j \to -\eps,
   &\quad& \text{as } j \to \infty.
\end{alignat*}
Letting $\eps \to 0$ gives $\lim_{j \to \infty} \oHp f(y_j) =0$, i.e.\ 
\ref{not-reg-one-seq} fails.

\ref{Cp-V-bdy} $\eqv$ \ref{semireg-local}
Note first that \ref{Cp-V-bdy} is equivalent to
the existence of a neighbourhood $W$ of $x_0$
with $\Cpx(W \cap \bdy G)=0$.
But this is equivalent to \ref{semireg-local},
by the already proved equivalence \ref{Cp-V-bdy} $\eqv$ \ref{semireg}
applied to $G$ instead of $\Om$.

\ref{Cp-V-bdy} $\imp$ \ref{Cp-V}
By Theorem~\ref{thm-irr-char-V}, \ref{V-Cp-V-bdy} $\imp$ \ref{V-alt-def-irr-super},
the set $\Om \cup (V \cap \bdy \Om)$ is open, and we
can use this as our set $V$ in \ref{Cp-V}.

\ref{Cp-V} $\imp$ \ref{Cp-V-bdy}
This is trivial.

\ref{Cp-V} $\eqv$ \ref{alt-def-irr} $\eqv$ 
\ref{alt-def-irr-super}
In all three statements 
it follows directly that $V \subset \overline{\Om}$.
Thus their equivalence follows directly from Theorem~\ref{thm-irr-char-V},
with $V$ in Theorem~\ref{thm-irr-char-V} corresponding to
$V \cap \bdy \Om$ here.

\ref{alt-def-irr} $\imp$ \ref{rem-irr}
The first part is obvious and we only need to show that
$x_0$ is irregular, 
but this follows from the already proved implication
\ref{alt-def-irr} $\imp$ \ref{semireg}

\ref{rem-irr} $\imp$ \ref{semireg}
Let $f\in C(\bdy \Om)$.
Then $\oHp f$ has a \px-harmonic extension $U$ to $\Om \cup V$
for some neighbourhood $V$ of $x_0$.
It follows that
\[
       \lim_{\Om \ni y \to x_0} \oHp f(y)=U(x_0),
\]
and thus the limit in the left-hand side always exists.
Since $x_0$ is irregular it follows that $x_0$ must be semiregular.
\end{proof}

We end this paper with some 
examples of semiregular and strongly irregular boundary points.

\begin{example} \label{ex-punctured-ball}
\textup{(The punctured ball)}
Let $\Om=B(x_0,r) \setm \{x_0\}$.
Then $x_0$ is semiregular if $\Cpx(\{x_0\})=0$,
and regular otherwise.
Indeed, when $\Cpx(\{x_0\})=0$ this follows from
Theorem~\ref{thm-irr-char-V}, and is also a special
case of Proposition~\ref{prop-ball-minus-K} below,
while if $\Cpx(\{x_0\})>0$ it follows
from the Kellogg property (Theorem~\ref{thm-kellogg}).
The remaining boundary points are all regular
by the sufficiency part of the Wiener criterion, or some weaker version of it.
\end{example}

\begin{prop} \label{prop-ball-minus-K}
Let $K$ be a compact set with $\Cpx(K)=0$.
Then there is a domain\/ $\Om$ such that
$K$ is the set of semiregular boundary points and all 
other boundary points are regular.
\end{prop}

\begin{proof}
Let $B$ be an open ball containing $K$ and let $\Om=B \setm K$.
Then $K$ is an open subset of $\bdy \Om$ and as $\Cpx(K)=0$ 
it follows from Theorem~\ref{thm-irr-char-V} that $K$ consists entirely
of semiregular points.
By the sufficiency part of the  Wiener criterion, or some weaker version of it,
all other boundary points
are regular. 
\end{proof}

\begin{prop} \label{prop-K1-K2}
Assume that $\pplus \le n$.
Let $K_1$ and $K_2$  be two disjoint compact subsets of\/ $\R^n$
with $\Cpx(K_1)=\Cpx(K_2)=0$.
Then there is a domain\/ $\Om$ such that
$K_1$ is the set of semiregular boundary points,
$K_2$ is the set of strongly irregular boundary points,
 and all 
other boundary points are regular.
\end{prop}

The proof of this is very similar to the proof of the corresponding 
result for the constant $p$ case, as given
for Theorem~4.1 in A.~Bj\"orn~\cite{ABclass},
and we leave it to the interested reader to verify.
Here we need to use the Wiener criterion. 
An essential fact also used in the proof is that points have zero capacity,
which is the reason for the requirement $\pplus \le n$.
Whether the result is true without this condition is not clear, 
see Section~5 in \cite{ABclass}.

\end{document}